\pdfoutput=1
\documentclass[a4paper, notitlepage, twoside]{article}
\usepackage{etex,ifdraft}
\usepackage[latin1]{inputenc}
\usepackage[english, british]{babel}
\usepackage{calc,amssymb,mathtools}
\mathtoolsset{mathic}
\usepackage[margin=10pt,font=small,labelfont=bf]{caption}
\usepackage{amsthm,stmaryrd,units,paralist,tabularx,booktabs,Dynkin,wrapfig,multicol,nicefrac,changepage,needspace,zref,linegoal}
\usepackage[hmarginratio=1:1,textwidth=12cm,top=3cm,bottom=3cm]{geometry}
\usepackage[final]{hyperref}
\hypersetup{
  bookmarks=false,
  breaklinks=true,
  colorlinks=true,
  linkcolor=black,
  anchorcolor=black,
  citecolor=black,
  filecolor=black,
  menucolor=black,
  pagecolor=black,
  urlcolor=black
}
\usepackage[alphabetic,non-sorted-cites]{amsrefs}
\usepackage[all]{xy}
\usepackage[nameinlink]{cleveref}
\chardef\bslash=`\\ 

\swapnumbers
\theoremstyle{plain} 
\newtheorem{thm}{Theorem}[section]\crefname{thm}{theorem}{theorems}
\newtheorem{thmVanishing}[thm]{Vanishing Theorem}\crefname{thmVanishing}{Vanishing Theorem}{}
\newtheorem{thmStructure}[thm]{Structure Theorem}\crefname{thmStructure}{Structure Theorem}{}
\newtheorem{thmPanin}[thm]{Panin's Theorem}\crefname{thmPanin}{Panin's Theorem}{}
\newtheorem{thmSteinberg}[thm]{Steinberg's Theorem}\crefname{thmSteinberg}{Steinberg's Theorem}{}
\newtheorem{thmSerre}[thm]{Serre's Theorem}\crefname{thmSerre}{Serre's Theorem}{}
\newtheorem*{thm*}{Theorem}
\newtheorem{cor}[thm]{Corollary}\crefname{cor}{corollary}{corollaries}
\newtheorem*{cor*}{Corollary}
\newtheorem{lem}[thm]{Lemma}\crefname{lem}{lemma}{lemmas}
\newtheorem*{obs*}{Observation}
\newtheorem*{lem*}{Lemma}
\newtheorem{prop}[thm]{Proposition}\crefname{prop}{proposition}{propositions}
\theoremstyle{definition}
\newtheorem{defn}[thm]{Definition}\crefname{defn}{definition}{definitions}
\newtheorem*{defn*}{Definition}
\theoremstyle{remark}
\crefname{example}{example}{examples}
\newtheorem*{example*}{Example}
\newtheorem{rem}[thm]{Remark}\crefname{rem}{remark}{remarks}
\newtheorem*{rem*}{Remark}
\crefname{equation}{eq.}{eq.}
\newcommand*{\Iff}{\Leftrightarrow}
\newcommand*{\N}{{\mathbb{N}}}
\newcommand*{\Z}{{\mathbb{Z}}}

\newcommand*{\Q}{{\mathbb{Q}}}

\renewcommand*{\bar}[1]{\smash{\overline{#1}}}

\newcommand{\factor}[2]{{\ensuremath{#1}/\ensuremath{#2}}}

\newcommand*{\A}{{\mathbb{A}}}

\renewcommand*{\P}{{\mathbb{P}}}
\DeclareMathOperator{\id}{id}
\DeclareMathOperator{\rank}{rank}
\renewcommand*{\rank}{\mathrm{rk}}
\renewcommand*{\mid}{\;\middle|\;}
\providecommand{\abs}[1]{\lvert#1\rvert}

\newcommand*{\cat}[1]{\mathcal{#1}}
\DeclareMathOperator{\cone}{cone}

\DeclareMathOperator{\im}{im}

\DeclareMathOperator{\Pic}{Pic}
\newcommand*{\twistspace}{\mathfrak L}

\newcommand*{\mm}[1]{\left(\begin{smallmatrix}#1\end{smallmatrix}\right)}
\renewcommand*{\matrix}[1]{\begin{pmatrix}#1\end{pmatrix}}

\newcolumntype{M}{>{$}c<{$}}
\entrymodifiers={+!!<0pt,\fontdimen22\textfont2>}
\SelectTips{cm}{10}

\newcommand*{\Grassmannian}[2]{\mathrm{Gr}_{#1,#2}}
\DeclareMathOperator{\Kgroup}{K}
\newcommand*{\K}{\ensuremath{\Kgroup}}

\DeclareMathOperator{\Wgroup}{W}
\newcommand*{\W}{\Wgroup}

\DeclareMathOperator{\GWgroup}{GW}
\newcommand*{\GW}{\GWgroup}

\newcommand*{\w}{\mathit{w}}

\newcommand*{\total}{\bullet}

\newcommand*{\lb}[1]{{\mathcal{#1}}}
\newcommand*{\vb}[1]{{\mathcal{#1}}}
\newcommand*{\dual}{*}

\renewcommand*{\star}{\ensuremath{\ast}}
\newcommand*{\ddual}{\circ}
\newcommand*{\ideal}[1]{\mathfrak{#1}}
\newcommand*{\Rep}{\mathrm{Rep}}

\renewcommand*{\vec}[1]{\mathbf{#1}}
\newcommand*{\mat}[1]{\mathrm{#1}}
\newcommand*{\WLattice}[1][]{{\Lambda_{#1}}}
\newcommand*{\Weyl}[1][]{{W_{#1}}}
\newcommand*{\FWC}[1][]{{\mathcal{W}_{#1}}}

\newcommand*{\simpleR}{\Sigma}

\newcommand*{\thetaR}{\Theta}
\newcommand*{\nthetaR}{{\simpleR\!-\!\thetaR}}
\newcommand*{\teta}{\vartheta}
\newcommand*{\RSystem}[1][]{{\mathcal R}_{#1}}

\newcommand*{\pairing}[2]{{\langle{#1,#2}\rangle}}

\newcommand*{\prii}{{\pi^{\scriptscriptstyle\perp}}}

\ifdraft{
\CompileMatrices
\CompilePrefix{XY/xymatrix}
\newcolumntype{M}{>{$}c<{$}}
}{}
\newcommand*{\deffont}[1]{\textbf{#1}}
\newcommand{\encouragepagebreak}{\needspace{0.2\textheight}}

\newcommand*{\GrothendieckWitt}{Grothen\-dieck-Witt }
\newcommand*{\ie}{\mbox{i.\thinspace{}e.\ }}
\newcommand*{\eg}{\mbox{e.\thinspace{}g.\ }}
\newcommand*{\cf}{\mbox{c.\thinspace{}f.\ }}
\begin{document}
\title{Twisted Witt Groups of Flag Varieties}
\author{Marcus Zibrowius%
  \thanks{Bergische Universität Wuppertal,
    Gaußstraße 20,
    42119 Wuppertal,
    Germany}
}
\date{\today}
\maketitle
\begin{abstract}
  Calmès and Fasel have shown that the twisted Witt groups of split flag varieties vanish in a large number of cases.
  For flag varieties over algebraically closed fields, we sharpen their result to an if-and-only-if statement.
  In particular, we show that the twisted Witt groups vanish in many previously unknown cases.
  In the non-zero cases,  we find that the twisted total Witt group forms a free module of rank one over the untwisted total Witt group, up to a difference in grading.

  Our proof relies on an identification of the Witt groups of flag varieties with the Tate cohomology groups of their K-groups, whereby the verification of all assertions is eventually reduced to the computation of the (twisted) Tate cohomology of the representation ring of a parabolic subgroup.
\end{abstract}
\tableofcontents
\thispagestyle{empty}
\newpage

\setlength{\parindent}{0pt}
\addtolength{\parskip}{3pt}
\addtolength{\topsep}{3pt}

\encouragepagebreak
\section*{Introduction}
Let \(G\) be a split semisimple algebraic group with a parabolic subgroup \(P\).
Our object of study is the Witt cohomology of the flag variety \(G/P\).
Specifically, we address the following question:
\begin{quote}
  Given the untwisted Witt groups \(\W^i(G/P)\), what can we say about the Witt groups \(\W^i(G/P,\lb L)\) with twists in a line bundle \(\lb L\) over \(G/P\)?
\end{quote}
A result of Calmès and Fasel shows that in many cases there is a short answer that does not even depend on the untwisted Witt groups:  the twisted Witt groups are simply zero \cite{CalmesFasel}.
The cases they deal with can be described as follows.
Suppose \(P\) is the standard parabolic subgroup corresponding to a subset \(\thetaR\) of a set of simple roots \(\simpleR\) of \(G\).
Consider the Dynkin diagram of \(G/P\), \ie the Dynkin diagram of \(G\) in which the roots contained in \(\thetaR\) and \(\nthetaR\) are indicated using different colours, say black for the roots in \(\thetaR\) and white for the remaining roots\footnote{This colouring scheme follows \cite{CalmesFasel}.  It is inverse to the colouring scheme used in \cite{Me:WCCV}*{Table~2} and \S~23.3 of \cite{FultonHarris}.
}.
For example, the Dynkin diagram of the Grassmannian of three-planes in eight-dimensional space would be:
\[\ifdraft{}{
    \begin{tikzpicture}[dynPicture]
      {[dynDiagram]
        {[dynTheta]
          \dynR{}
          \dynR{}
        }
        \dynR{}
        {[dynTheta]
          \dynR{}
          \dynR{}
          \dynR{}
          \dynR{}
        }
      }
    \end{tikzpicture}
}\]
Then there is a natural one-to-one correspondence between the white nodes and a basis of the Picard group modulo squares \(\Pic(G/P)/2\).
Let \(\mathfrak N\subset \Pic(G/P)/2\) be the linear subspace generated by the basis elements corresponding to those white nodes that have a black neighbour, \ie that are connected to at least one black node by an edge of the Dynkin diagram.
Calmès and Fasel's result can be paraphrased as follows:
\begin{quote}
  If the class of \(\lb L\) in \(\Pic(G/P)/2\) lies outside the subspace \(\mathfrak N\), then \(\W^i(\Pic(G/P),\lb L)\) vanishes for all \(i\).
\end{quote}
For example, for a full flag variety \(G/B\) with \(B\subset G\) a Borel subgroup, the Dynkin diagram of \(G/B\) has white nodes only, the subspace \(\mathfrak N\) is trivial and the result therefore says that \emph{all} twisted Witt groups vanish.  This has been used by Yagita to identify the untwisted Witt ring \(\W^\total(G/B)\) with the Witt ring of \(G\) \cite{Yagita:W-of-G}*{Theorem~1.1}.
However, although Calmès and Fasel's result covers a great many cases, it is not as sharp as possible. For example, as the authors themselves note, it does not predict the vanishing of all twisted Witt groups of the Grassmannian above. (They are known to vanish by \cite{BalmerCalmes}.)

In this paper, we sharpen Calmès and Fasel's criterion under the assumption that the ground field is algebraically closed as follows:  we describe generators of a smaller linear subspace \(\twistspace\subset\mathfrak N\) such that the twisted total Witt group \(\W^\total(G/P,\lb L) := \bigoplus\W^i(G/P,\lb L)\) is non-zero if and only if the reduction of \(\lb L\) modulo two is contained in \(\twistspace\).  In addition to the Grassmannian above, here are the Dynkin diagrams of some further examples for which \(\twistspace\), unlike \(\mathfrak N\), turns out to be trivial:
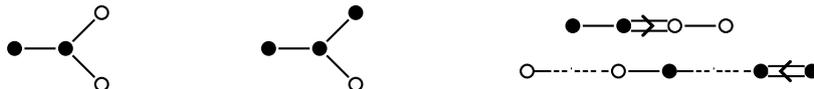
\[
\begin{aligned}
  \begin{tikzpicture}[dynPicture]
    {[dynDiagram]
      \dynR{[theta] }
      \dynR{[theta] (v)}
      \dynR{[above right=of v]}
      \dynR[with v]{[below right=of v]}
    }
  \end{tikzpicture}
  \hspace{2cm}
  \begin{tikzpicture}[dynPicture]
    {[dynDiagram]
      \dynR{[theta] }
      \dynR{[theta] (v)}
      \dynR{[above right=of v][theta]}
      \dynR[with v]{[below right=of v]}
    }
  \end{tikzpicture}
\end{aligned}
\hspace{2cm}
\begin{aligned}
  \begin{tikzpicture}[dynPicture]
    {[dynDiagram]
      \dynR{[empty] (a)}
    }
    {[dynDiagram]
      \dynR{[theta][right=of a]}
      \dynR{[theta] }
      \dynR[by right arrow]{ }
      \dynR{ }
    }
    {[dynDiagram]
      \dynR{[below=of a]}
      \dynR[by fadeout line]{[empty] }
      \dynR[by fadein line]{ }
      {[dynTheta]
        \dynR{}
        \dynR[by fadeout line]{[empty] }
        \dynR[by fadein line]{ }
        \dynR[by left arrow]{ }
      }
    }
  \end{tikzpicture}
\end{aligned}
\]
We also  have the following structural result concerning the non-zero cases:
\begin{thm*}
  If one of the twisted groups \(\W^i(G/P,\lb L)\) is non-zero, then there exists an element \(\zeta\in\W^0(G/P,\lb L) \oplus \W^2(G/P,\lb L)\) such that multiplication with \(\zeta\) induces an isomorphism
  \[
  \W^\total\left(\factor{G}{P}\right) \xrightarrow{\cong} \W^\total\left(\factor{G}{P},\lb L\right).
  \]
\end{thm*}
The element \(\zeta\) is not always homogeneous, so this isomorphism does not necessarily preserve the \(\Z/4\)-grading of the total Witt group.  It only preserves the \(\Z/2\)-grading consisting of an even part \(\W^0\oplus\W^2\) and an odd part \(\W^1\oplus\W^3\).

Our proof exploits the relationship between Witt groups and the Tate cohomology of K-groups advocated in \cite{Me:KOFF}.
Given a variety \(X\), let \(\star\) denote the involution on its algebraic K-group \(\K_0(X)\) induced by the duality on vector bundles, and let
\begin{align*}
   h^+(X) &:=\frac{\ker(\id-\;\star)}{\im(\id+\;\star)}&
   h^-(X) &:=\frac{\ker(\id+\;\star)}{\im(\id-\;\star)}
\end{align*}
denote the Tate cohomology of \((\K_0(X),\star)\).  We show:
\begin{thm*}
  For any smooth cellular variety over an algebraically closed field of characteristic not two, we have isomorphisms
  \begin{align*}
    h^+(X) &\cong \W^0(X) \oplus \W^2(X) \\
    h^-(X) &\cong \W^1(X) \oplus \W^3(X)
  \end{align*}
\end{thm*}
This theorem and some related results of \Cref{sec:W=h} have already been used in \cite{Me:KOFF} to obtain a description of the untwisted Witt rings of full flag varieties.  Unfortunately, the assumption that the ground field is algebraically closed is essential in our approach: there is no chance for the relationship between Witt groups and Tate cohomology to hold over an arbitrary field in the form stated as Tate cohomology is always two-torsion.
This is the reason why said assumption enters all results on Witt groups presented here.
However, each result has a precursor in terms of Tate cohomology which we prove without this assumption.

Apart from being valid over arbitrary fields, Calmès and Fasel's approach has another competitive advantage: not only is it much easier to state, but also the whole proof fits onto two pages.
However, it should be noted that a key input of their argument is an existing `vanishing result' given by Walter's calculations of the Witt groups of projective bundles \cite{Walter:PB}.  The approach presented here does not presuppose any existing Witt group computations.

\subsection*{Acknowledgements}
I warmly thank Baptiste Calmès and Jean Fasel for many interesting discussions and for sharing their ideas on the problem considered here.  In particular, their unpublished computations of Witt groups through cellular spectral sequences have helped me track down mistakes in an earlier version of this article. They also have an independent proof of the above theorem relating Witt groups and Tate cohomology.

\encouragepagebreak
\section{Statement of the results}\label{sec:results}
\subsection{A structure theorem}\label{sec:dichotomy}
Given a smooth variety \(X\) over a field of characteristic not two, we write \(\W^i(X)\) to denote the \(i^\text{th}\) shifted Witt group of the derived category of vector bundles over \(X\) equipped with the usual duality \(\vb E\mapsto \vb Hom(\vb E,\vb O_X)\) \cites{Balmer:TWGI,Balmer:TWGII}.
More generally, given a line bundle \(\lb L\) over \(X\), we write \(\W^i(X,\lb L)\) to denote the corresponding `twisted' Witt groups of \(X\), \ie the Witt groups of the same category equipped with the duality \(\vb E\mapsto \vb Hom(\vb E,\lb L)\).
As there are canonical isomorphisms \(\W^i(X,\lb L)\cong\W^{i+4}(X,\lb L)\), we will index these groups with \(i\in\Z/4\) in the following. We define the total \(\lb L\)-twisted Witt group of \(X\) as the direct sum
\[
\W^\total(X,\lb L) := \bigoplus_{i\in\Z/4} \W^i(X,\lb L).
\]
Again, when \(\lb L = \lb O_X\), we simply write \(\W^\total(X)\).
These total Witt groups obviously have a natural \(\Z/4\)-grading.
However, we will sometimes forget half of this grading and consider them to be only \(\Z/2\)-graded with homogeneous parts \(\W^{0,2}\) and \(\W^{1,3}\) defined in the obvious way:
\begin{align*}
   \W^{0,2}(X,\lb L) &:= \W^0(X,\lb L)\oplus \W^2(X,\lb L) \\
   \W^{1,3}(X,\lb L) &:= \W^1(X,\lb L)\oplus \W^3(X,\lb L)
\end{align*}
For any line bundles \(\lb L\) and \(\lb M\), we have a pairing
\[
\W^i(X,\lb L) \otimes \W^j(X,\lb M) \rightarrow \W^{i+j}(X,\lb L\otimes\lb M)
\]
such that \(\W^\total(X)\) is a (\(\Z/4\)- or \(\Z/2\)-)graded ring over which the twisted total Witt groups are graded modules \cite{GilleNenashev}.  It follows easily from the definitions that any line bundle \(\lb L\) over \(X\) defines a class \([\lb L]\) in \(\W^0(X,\lb L^2)\) such that multiplication with \([\lb L]\) induces an isomorphism
\[
\W^\total(X)\xrightarrow{\cong} \W^\total(X,\lb L^2)
\]
and more generally isomorphisms \(\W^\total(X,\lb N)\cong \W^\total(X,\lb N\otimes \lb L^2)\).
In particular, up to (non-unique) isomorphism, the \(\lb L\)-twisted Witt groups of \(X\) only depend on the class of \(\lb L\) in \(\Pic(X)/2\).

For twists by a line bundle \(\lb L\) that cannot be expressed as a tensor square, we are not aware of any general results concerning the relationship of \(\W^\total(X,\lb L)\) to \(\W^\total(X)\).
However, it turns out that when \(X\) is a flag variety, there are only two possibilities: either the twisted total Witt group vanishes, or it is free of rank one as a \(\Z/2\)-graded module over \(\W^\total(X)\).
Here and in the following, by a flag variety we mean a quotient \(G/P\) of a split semisimple algebraic group \(G\) by a parabolic subgroup \(P\).

\begin{thmStructure}\label{mainthm:iso}
  Let \(\lb L\) be a line bundle over a flag variety \(G/P\) over an algebraically closed field of characteristic not two.
  If the twisted total Witt group \(\W^\total(G/P,\lb L)\) is non-zero, then there exists an element \(\zeta_{\lb L}\) in \(\W^{0,2}(G/P,\lb L)\) such that multiplication by \(\zeta_{\lb L}\) induces an isomorphism of \(\Z/2\)-graded modules
  \[
  \W^\total\left(\factor{G}{P}\right) \xrightarrow[\cdot \zeta_{\lb L}]{\cong} \W^\total\left(\factor{G}{P},\lb L\right).
  \]
  Moreover, if both \(\W^\total(G/P,\lb L_1)\) and \(\W^\total(G/P,\lb L_2)\) are non-zero for line bundles \(\lb L_1\) and \(\lb L_2\), then so is \(\W^\total(G/P,\lb L_1\otimes \lb L_2)\) and we may choose \(\zeta_{\lb L_1\otimes \lb L_2} = \zeta_{\lb L_1}\cdot \zeta_{\lb L_2}\).
\end{thmStructure}
In particular, in the cases when the total \(\lb L\)-twisted Witt group is non-zero, we have isomorphisms
\[
\W^{0,2}(G/P)\cong \W^{0,2}(G/P,\lb L) \quad \text{ and } \quad \W^{1,3}(G/P)\cong \W^{1,3}(G/P,\lb L).
\]
The failure of these isomorphisms to respect the finer \(\Z/4\)-grading can already be seen for the most simple examples of flag varieties, as illustrated by the last line of \Cref{table:Gr}.
\begin{table}[b!t!]
  \begin{center}
    \newlength{\testheight}
    \settoheight{\testheight}{\( \dfrac{b}{b} \)}
    \begin{tabular}{l|MMMM|MMMM}
      \toprule
                            & \multicolumn{4}{M|}{\W^\total(X) } & \multicolumn{4}{M}{\W^\total(X,\lb L)}\\
       \(X\)           & {\W^0}     & {\W^1}  & {\W^2} & {\W^3} & {\W^0} & {\W^1} & {\W^2}     & {\W^3} \\
      \midrule
      \(\P^1               \) & \Z/2     & \Z/2  & 0    & 0    & 0    & 0    & 0        & 0   \\
      \(\P^4               \) & \Z/2     & 0     & 0    & 0    & \Z/2 & 0    & 0        & 0   \\
      \(\P^2               \) & \Z/2     & 0     & 0    & 0    & 0    & 0    & \Z/2     & 0   \\
      \(\Grassmannian{2}{4}\) & (\Z/2)^3 & 0     & 0    & 0    & \Z/2 & 0    & (\Z/2)^2 & 0   \\
      \bottomrule
    \end{tabular}
    \caption{The Witt groups of some projective spaces and of the Grassmannian of planes in six-dimensional space, over an algebraically closed field. In each case, the Picard group is free abelian on one generator \(\lb L\).}
    \label{table:Gr}
  \end{center}
\end{table}
\begin{rem}
  Essentially, \Cref{mainthm:iso} says that the total Witt group \(\W^\total(G/P,\lb L)\) behaves like the space of solutions to a linear system of equations with constant term \(\lb L\).  This is no coincidence: our proof will eventually reduce to this situation.
\end{rem}

\subsection{A vanishing theorem}\label{sec:vanishing}
We will describe for which line bundles the twisted Witt groups vanish using a graphical notation that we now explain.
Let \(T\) be a split maximal torus of \(G\), and let \(X^* := X^*(T)\) be its character group, the weight lattice of \(G\).
We write \(R\subset X^*\) for the set of roots of \(G\), and we chose a set of simple roots \(\simpleR\).

Note first that any flag variety under \(G\) is also a flag variety under its simply connected cover. So we may and will assume that \(G\) is simply connected.  Simply connected semisimple groups \(G\) are determined up to isomorphism their Dynkin diagram.

Next, we have a one-to-one correspondence between subsets \(\thetaR\subset\simpleR\) and conjugacy classes of parabolic subgroups of \(G\).
Thus, we can specify \(G/P\) up to isomorphism by drawing the Dynkin diagram of \(G\) with those nodes corresponding to the roots in \(\thetaR\) coloured, as in the introduction. For example, the diagram
\[\ifdraft{}{
  \begin{tikzpicture}[dynPicture]
    {[dynDiagram]
      \dynR{}
      {[dynTheta]
        \dynR{}
      }
      \dynR{}
      \dynR{}
      \dynR{}
      {[dynTheta]
        \dynR{}
      }
      \dynR[by right arrow]{}
    }
  \end{tikzpicture}
}\]
will denote the quotient of the simply connected semisimple group of type \(B_7\) by a parabolic subgroup corresponding to a two-element subset of \(\simpleR\).
The subdiagram consisting of all black nodes and edges between them is the Dynkin diagram of the Levi subgroup of \(P\) (\cf \Cref{sec:rep-of-P}).

A standard representative of the conjugacy class corresponding to a subset \(\thetaR\subset\simpleR\) is given by the subgroup \(P_\thetaR\) of \(G\) generated by the maximal torus \(T\), the root subgroups of all negative roots, and the root subgroups of all positive roots contained in \(\Z\thetaR\).  We will assume from now on that \(P\) is of this form.

The Picard group of \(G/P\) is a free abelian group on line bundles \(\lb L_\beta\) corresponding to the simple roots in \(\nthetaR\), \ie to the white nodes of our diagram.
This correspondence can be described as follows: Let \(\left\{\omega_\sigma\mid\sigma\in\simpleR\right\}\) be the set of fundamental weights.  As we have assumed \(G\) to be simply connected, these are actual characters of \(T\) and not just elements of \(X^*\otimes\Q\).  Those fundamental weights \(\omega_\beta\) with \(\beta\in\nthetaR\) extend to characters of \(P = P_\thetaR\) and hence define line bundles \(\lb L_{\beta}\) over \(G/P\).

We can thus specify a line bundle \(\lb L_\beta\) over \(G/P\) by drawing a `line mark' over the node corresponding to \(\beta\) in the Dynkin diagram of \(G/P\), like so:
\[\ifdraft{}{
  \begin{tikzpicture}[dynPicture]
    {[dynDiagram]
      \dynR{}
      \dynR{}
      \dynR{}
      \dynR{}
      \dynR{ (c) }
      {[dynTheta]
        \dynR{}
      }
      \dynR[by right arrow]{}
    }
    \dyntwist{c}
  \end{tikzpicture}
}\]
More generally, we will use a line mark connecting several such nodes to specify a tensor product of the corresponding line bundles.
For example,
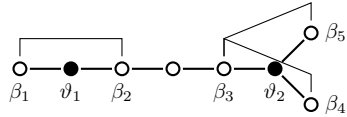
\[\ifdraft{}{
  \begin{tikzpicture}[dynPicture]
    {[dynDiagram]
      \dynR{[lbl=$\beta_1$] (b1) }
      \dynR{[theta][lbl=$\teta_1$] }
      \dynR{[lbl=$\beta_2$] (b2)}
      \dynR{ }
      \dynR{[lbl=$\beta_3$] (b3)}
      \dynR{[theta][lbl=$\teta_2$] (v)}
      \dynR{[above right=of v][label={right:$\beta_5$}] (b5)}
      \dynR[with v]{[below right=of v][label={right:$\beta_4$}] (b4)}
    }
    \dyndoubletwist{b1}{b2}
    \dyntripletwist{b4}{b3}{b5}
  \end{tikzpicture}
}\]
will denote the line bundles \mbox{\(\lb L_{\beta_1}\otimes \lb L_{\beta_2}\)} and \(\lb L_{\beta_3}\otimes \lb L_{\beta_4} \otimes \lb L_{\beta_5}\) on \(G/P_{\{\teta_1,\teta_2\}}\), respectively.

By the second part of \Cref{mainthm:iso}, the classes in \(\Pic(G/P)/2\) of all line bundles \(\lb L\) for which \(W^*(G/P,\lb L)\) is non-zero form a linear subspace.
The idea of our vanishing theorem is to describe a set of generators of this subspace by specifying the corresponding line marks of the Dynkin diagram of \(G/P\).  The precise rules for obtaining these marks are laid out in the next section.
\begin{thmVanishing}\label{mainthm:vanishing}
  Let \(\lb L\) be a line bundle over a flag variety \(G/P\) as in \Cref{mainthm:iso}.
  Assume without loss of generality that \(G\) is simply connected, and that we have chosen a set of simple roots \(\simpleR\) such that \(P\) is the standard parabolic subgroup corresponding to a subset \(\thetaR\subset\simpleR\).
  Then the \(\lb L\)-twisted total Witt group \(\W^\total(G/P,\lb L)\) is non-zero if and only if the reduction of \(\lb L\) modulo two is a linear combination of the line bundles marked in the Dynkin diagram of \(G/P\) according to the marking scheme below.
\end{thmVanishing}

\subsection{A marking scheme}\label{sec:marking-scheme}
Let \(G/P\) and \(\thetaR\subset\simpleR\) be as in \Cref{mainthm:iso}.  We use the same letters to denote the corresponding Dynkin diagrams, \ie \(\simpleR\) for the Dynkin diagram of \(G\) and \(\thetaR\) for the subdiagram spanned by the black nodes.  Each connected component \(\thetaR_0\) of \(\thetaR\) determines line marks only on the connected component of \(\simpleR\) in which it is contained, so we may assume that \(\simpleR\) is connected.  Moreover, the marks contributed by \(\thetaR_0\) do not depend on any other components of \(\thetaR\), so we can lay out the rules for marking the diagram one connected component \(\thetaR_0\) at a time.
We organize these rules according to the type of the ambient diagram \(\simpleR\), illustrating each case with a short list of examples.

\newcommand{\Luft}{\vskip 10pt plus 60pt minus 10pt\pagebreak[3]}
\Luft
\begin{description}
  \newcommand{\pfusch}{$ $\\\nopagebreak\vspace{-12pt}\nopagebreak} \newcommand{\mbf}[1]{\ensuremath{\pmb{#1}}} \newcommand*{\outerdiag}[1]{\mbf{\simpleR} of type \mbf{#1}} \newcommand*{\innerdiag}[2]{\mbf{\thetaR_0} of type \mbf{#1}}
\item[\outerdiag{A_n}]\hfill\pfusch
\begin{asparadesc}
  \item[\innerdiag{A_l}{A_n} with \mbf{l} odd]$ $\\
    Connect all neighbours of \(\thetaR_0\) with a single mark. That is:\\
    If \(\thetaR_0\) has a unique neighbour, mark that neighbour.\\
    If \(\thetaR_0\) has two neighbours, connect these with a mark.
    \vskip \plitemsep

    (Subdiagrams of type $A_l$ with $l$ even do contribute any marks.)
  \end{asparadesc}
  \begin{multicols}{3}
    \begin{tikzpicture}[dynPicture,auto]
      {[dynDiagram] \dynR{ } \dynR{ (a) } \dynR{[theta]} \dynR{ (c) } \dynR{ } \dyndoubletwist{a}{c} }
    \end{tikzpicture}

    \begin{tikzpicture}[dynPicture,auto]
      {[dynDiagram] \dynR{ (a) } \dynR{[theta]} \dynR{[theta]} \dynR{ } \dynR{ } \dyntwist[white]{a}
      }
    \end{tikzpicture}

    \begin{tikzpicture}[dynPicture,auto]
      {[dynDiagram] \dynR{[theta]} \dynR{[theta]} \dynR{[theta]} \dynR{ (d) } \dynR{ } \dyntwist{d} }
    \end{tikzpicture}
  \end{multicols}
  \Luft

\item[\outerdiag{B_n}]\hfill\pfusch
  \begin{asparadesc}
  \item[\innerdiag{A_l}{B_n} with \mbf{l} odd]$ $\\
    Connect all neighbours of \(\thetaR_0\) with a single mark, except for the unique shortest root \(\sigma_n\) of \(\simpleR_0\).  This root is never marked.
  \item[\innerdiag{B_l}{B_n}] Mark the unique neighbour (for \(l\neq n\)).
    \vskip \plitemsep

    (Subdiagrams of type $A_l$ with $l$ even do not contribute any marks.)
  \end{asparadesc}
  \begin{multicols}{2}
    \ifdraft{}{
      \begin{tikzpicture}[dynPicture]
        {[dynDiagram] \dynR{ (a) } {[dynTheta] \dynR{} } \dynR{ (b) } \dynR{} \dynR{ (c) } {[dynTheta] \dynR{} } \dynR[by right arrow]{[lbl=$\sigma_n$]} } \dyndoubletwist{a}{b} \dyntwist{c}
      \end{tikzpicture}

      \begin{tikzpicture}[dynPicture]
        {[dynDiagram] {[dynTheta] \dynR{} \dynR{} } \dynR{} \dynR{ (c) } \dynR{[theta] } \dynR{[theta]} \dynR[by right arrow]{[theta][lbl=$\sigma_n$]} } \dyntwist{c}
      \end{tikzpicture}
    }
  \end{multicols}
  \Luft

\item[\outerdiag{C_n}]\hfill\pfusch
  \begin{asparadesc}
  \item[\innerdiag{A_l}{C_n} with \mbf{l} odd]$ $\\
    Connect all neighbours of \(\thetaR_0\) with a single mark.
    \vskip \plitemsep

    (Subdiagrams of types $A_l$ with $l$ even, $B_2$ and $C_l$ do not contribute any marks.)
  \end{asparadesc}
  \begin{multicols}{2}
    \ifdraft{}{
      \begin{tikzpicture}[dynPicture]
        {[dynDiagram] \dynR{ (a) } {[dynTheta] \dynR{} } \dynR{ (b) } \dynR{} \dynR{ (c) } {[dynTheta] \dynR{} } \dynR[by left arrow]{ (d) } } \dyndoubletwist{a}{b} \dyndoubletwist{c}{d}
      \end{tikzpicture}

      \begin{tikzpicture}[dynPicture]
        {[dynDiagram] {[dynTheta] \dynR{} \dynR{} } \dynR{} \dynR{ (c) } \dynR{[theta] } \dynR{[theta]} \dynR[by left arrow]{[theta]} } \dyntwist[white]{c}
      \end{tikzpicture}
    }
  \end{multicols}
  \Luft

\item[\outerdiag{D_n}]\hfill\pfusch
  \begin{asparadesc}
  \item[\innerdiag{A_l}{D_n} with \mbf{l} odd]$ $\\
    Connect all neighbours of the two outer roots of \(\thetaR_0\) with a mark.\\ (For \(l=1\), connect all neighbours.)
  \item[\innerdiag{D_l}{D_n} with \mbf{l} even] Mark the unique neighbour.
    \vskip \plitemsep

    (Subdiagrams of types $A_l$ with $l$ even and $D_l$ with $l$ odd do not contribute any marks.)
  \end{asparadesc}
  \begin{multicols}{3}
    \ifdraft{}{
      \begin{tikzpicture}[dynPicture]
        {[dynDiagram] \dynR{ (a) } \dynR{[theta]} \dynR{[theta]} \dynR{[theta] (v)} \dynR{[above right=of v] (vv)} \dynR[with v]{[below right=of v] (w)} \dyntripletwist{vv}{a}{w} }
      \end{tikzpicture}

      \begin{tikzpicture}[dynPicture]
        {[dynDiagram] \dynR{ } \dynR{ (a) } \dynR{[theta]} \dynR{[theta] (v)} \dynR{[theta][above right=of v] (vv)} \dynR[with v]{[below right=of v] (w)} \dyntwist{a} \dyntwist[white]{vv}
        }
      \end{tikzpicture}

      \begin{tikzpicture}[dynPicture]
        {[dynDiagram] \dynR{ } \dynR{ (a) } \dynR{[theta]} \dynR{[theta] (v)} \dynR{[theta][above right=of v] (vv)} \dynR[with v]{[theta][below right=of v] (w)} \dyntwist[white]{vv}
        }
        \dyntwist{a}
      \end{tikzpicture}
    }
  \end{multicols}
  \Luft

\item[\outerdiag{E_n} (\mbf{n\in\{6,7,8\}})]\hfill\pfusch
  \begin{asparadesc}
  \item[\innerdiag{A_l}{E_n} with \mbf{l} odd]$ $\\\noindent
    If \(l=1\), connect all neighbours of \(\thetaR_0\) with a mark.\\
    If \(l=3\), connect all neighbours of the two outer roots of \(\thetaR_0\) with a mark.\\
    If \(l=5\), connect all neighbours of the two outer roots and of the central root of \(\thetaR_0\) with a mark.\\
    If \(l=7\) and \(n=8\), mark the unique neighbour of \(\thetaR_0\).
  \item[\innerdiag{D_l}{E_n}]$ $\\\noindent
    If \(l=4\), \(l=6\) or \(l=7\), mark all neighbours of \(\thetaR_0\) individually.\\
    \parbox{-3cm+\linegoal}{ If \(l=5\), one of the roots labeled \(\teta_4\) and \(\teta_5\) in the diagram on the right will have a unique neighbour in \(E_n\).  Mark that neighbour.  }
    \hfill
    \parbox{2.5cm}{\raggedright
      \ifdraft{\\-- a figure--}{
        \mbox{}\hfill
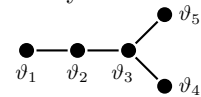
        \begin{tikzpicture}[dynPicture]
          {[dynDiagram, dynTheta]
            \dynR{[lbl=$\Huge\teta_1$]}
            \dynR{[lbl=$\teta_2$]}
            \dynR{[lbl=$\teta_3\;\;$](v)}
            \dynR{[label={right:$\teta_5$}][above right=of v] (vv)}
            \dynR[with v]{[label={right:$\teta_{4}$}][below right=of v] (w)}
          }
        \end{tikzpicture}
      } }
  \item[\innerdiag{E_l}{E_n}]$ $\\\noindent
    If \(l=7\) and \(n=8\), mark the unique neighbour of \(\thetaR_0\).

    \vskip \plitemsep
    (Subdiagrams of types $A_l$ with $l$ even and $E_6$ do not contribute any marks.)
  \end{asparadesc}

  \begin{multicols}{2}
    \ifdraft{}
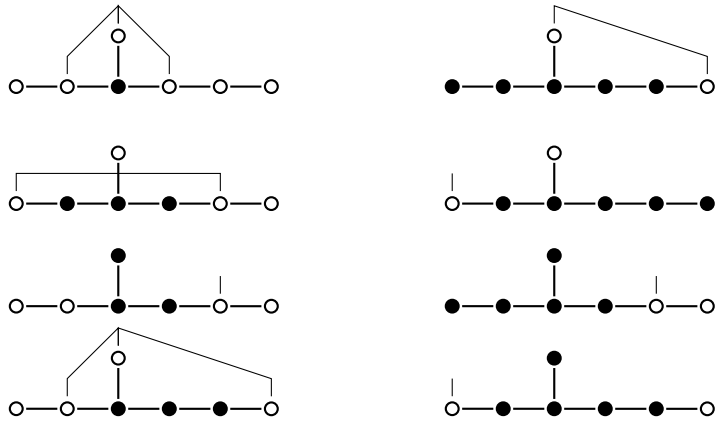
{
      \begin{tikzpicture}[dynPicture]
        {[dynDiagram]
          \dynR{(pos1) }
          \dynR{ (a8) }
          \dynR{[theta] (e68)}
          \dynR[with e68 by line]{[above=of e68] (b8)}
          \dynR[with e68 by line]{[right=of e68] (c8)}
          \dynR{ }
          \dynR{ } }
        \dyntripletwist{a8}{b8}{c8}
        {[dynDiagram]
          \dynR{[below=\Platz of pos1] (a7) }
          \dynR{[theta] }
          \dynR{[theta] (e67) }
          \dynR[with e67 by line]{[above=of e67]}
          \dynR[with e67 by line]{[right=of e67][theta]}
          \dynR{ (b7) } \dynR{ } }
        \dyndoubletwist{a7}{b7}
        {[dynDiagram]
          \dynR{[below=2\Platz of pos1] }
          \dynR{ }
          \dynR{[theta] (e61) }
          \dynR[with e61 by line]{[above=of e61][theta]}
          \dynR[with e61 by line]{[right=of e61][theta]}
          \dynR{ (a1) }
          \dynR{ } }
        \dyntwist{a1}
         {[dynDiagram]
          \dynR{[below=3\Platz of pos1] }
          \dynR{ (a5) }
          \dynR{ (e65) [theta]}
          \dynR[with e65 by line]{[above=of e65] (b5) }
          \dynR[with e65 by line]{[right=of e65][theta]}
          \dynR{[theta]}
          \dynR{ (c5) } }
        \dyntripletwist{a5}{b5}{c5}
      \end{tikzpicture}
      \columnbreak

      \begin{tikzpicture}[dynPicture]
        {[dynDiagram]
          \dynR{at (0,0) [theta] (pos1)}
          \dynR{[theta]} \dynR{[theta] (e62) }
          \dynR[with e62 by line]{[above=of e62] (a2)}
          \dynR[with e62 by line]{[right=of e62][theta]}
          \dynR{[theta]} \dynR{ (b2) }
          \dyndoubletwist{a2}{b2} }
        {[dynDiagram]
          \dynR{[below=\Platz of pos1] (a3) }
          \dynR{[theta]} \dynR{[theta] (e63) }
          \dynR[with e63 by line]{[above=of e63]}
          \dynR[with e63 by line]{[right=of e63][theta]}
          \dynR{[theta]} \dynR{[theta]} }
        \dyntwist{a3}
        {[dynDiagram]
          \dynR{[theta][below=2\Platz of pos1]  }
          \dynR{[theta] } \dynR{[theta] (e64) }
          \dynR[with e64 by line]{[theta][above=of e64]}
          \dynR[with e64 by line]{[theta][right=of e64]} \dynR{ (b4) }
          \dynR{ } }
         \dyntwist{b4}
        {[dynDiagram]
          \dynR{[below=3\Platz of pos1] (a6)}
          \dynR{[theta] } \dynR{[theta] (e66) }
          \dynR[with e66 by line]{[above=of e66][theta]}
          \dynR[with e66 by line]{[right=of e66][theta]}
          \dynR{[theta]} \dynR{ }
          \dyntwist{a6} }
      \end{tikzpicture}
    }
  \end{multicols}
  \Luft

\item[\outerdiag{F_4}]
  Distribute marks according to the following diagrams:
  \begin{multicols}{3}\raggedcolumns
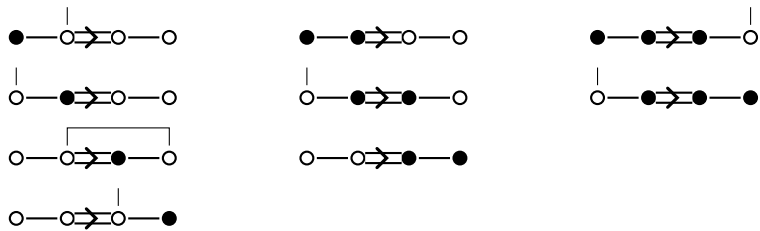
    \begin{tikzpicture}[dynPicture]
      {[dynDiagram] \dynR{[theta] at (0,0)} \dynR{ (t) } \dynR[by right arrow]{} \dynR{} \dyntwist{t} } {[dynDiagram] \dynR{at (0,-1) (t) } \dynR{[theta]} \dynR[by right arrow]{} \dynR{} \dyntwist{t} } {[dynDiagram] \dynR{at (0,-2)} \dynR{(t)} \dynR[by right arrow]{[theta]} \dynR{(tt)} \dyndoubletwist{t}{tt} } {[dynDiagram] \dynR{at (0,-3)} \dynR{} \dynR[by right arrow]{ (t) } \dynR{[theta]} \dyntwist{t} }
    \end{tikzpicture}
    \columnbreak

    \begin{tikzpicture}[dynPicture]
      {[dynDiagram] \dynR{at (0, 0) (a) [theta]} \dynR{[theta]} \dynR[by right arrow]{       } \dynR{} \dyntwist[white]{a}  }
      {[dynDiagram] \dynR{at (0,-1) (t)        } \dynR{[theta]} \dynR[by right arrow]{[theta]} \dynR{} \dyntwist{t} }
      {[dynDiagram] \dynR{at (0,-2)            } \dynR{       } \dynR[by right arrow]{[theta]} \dynR{[theta]} }
    \end{tikzpicture}
    \columnbreak

    \begin{tikzpicture}[dynPicture] {[dynDiagram] \dynR{[theta] at (0,0)} \dynR{[theta]} \dynR[by right arrow]{[theta]} \dynR{(t)} \dyntwist{t} } {[dynDiagram] \dynR{at (0,-1) (t)} \dynR{[theta]} \dynR[by right arrow]{[theta]} \dynR{[theta]} \dyntwist{t} }
    \end{tikzpicture}
  \end{multicols}
  \Luft

\item[\outerdiag{G_2}]\hfill\pfusch
  \begin{asparadesc}
  \item[\mbf{A_1\subset G_2}] Mark the unique neighbour:\phantom{Platz}
    \begin{tikzpicture}[dynPicture]
      {[dynDiagram] \dynR{ (a) } \dynR[by triple right arrow]{[theta]} } \dyntwist{a}
    \end{tikzpicture}
    \phantom{Platz}
    \begin{tikzpicture}[dynPicture]
      {[dynDiagram] \dynR{[theta]} \dynR[by triple right arrow]{ (b) } } \dyntwist{b}
    \end{tikzpicture}
  \end{asparadesc}
\end{description}
\Luft

When \(\thetaR\) has several connected components, the resulting marked diagram may sometimes be simplified.  After all, we are only interested in the subspace of \(\Pic(G/P)\) generated by the marked line bundles rather than in a particular choice of generators.  For example, the following diagram marked according to the above scheme
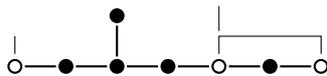
\[
  \begin{tikzpicture}[dynPicture,auto]
    {[dynDiagram]
      \dynR{ (a) }
      \dynR{[theta]}
      \dynR{[theta] (v)}
      \dynR{[theta][above=of v]}
      \dynR[with v by line]{[theta][right=of v]}
      \dynR{ (b) }
      \dynR{[theta]}
      \dynR{ (c) }
    }
    \dyntwist{a}
    \dyntwist{$(b)+(0,0.5)$}
    \dyndoubletwist{b}{c}
  \end{tikzpicture}
\]
may be simplified to:
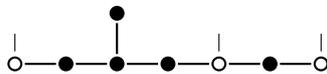
\[
\begin{tikzpicture}[dynPicture,auto]
  {[dynDiagram]
      \dynR{ (a) }
      \dynR{[theta]}
      \dynR{[theta] (v)}
      \dynR{[theta][above=of v]}
      \dynR[with v by line]{[theta][right=of v]}
      \dynR{ (b) }
      \dynR{[theta]}
      \dynR{ (c) }
    }
    \dyntwist{a}
    \dyntwist{b}
    \dyntwist{c}
\end{tikzpicture}
\]
Note that all diagrams  in the introduction remain undecorated according to the above marking scheme.

\encouragepagebreak
\section{From Witt rings to representation rings}\label{sec:W-to-Rep}
In this section we explain how we translate the computation of Witt rings of flag varieties to a computation with representation rings.
First, in \Cref{sec:W=h},  we show that the Witt ring of a smooth cellular variety over an algebraically closed field may be identified with the Tate cohomology of its K-ring.
Here, by a cellular variety we mean a variety \(X\) which can be equipped with a filtration by closed subsets
\[
X = X^0 \supset X^1 \supset X^2 \supset \cdots
\]
such that each complement \(X^i-X^{i+1}\) is isomorphic to affine space \(\A^{n_i}\) for some \(n_i\).
Split flag varieties are among the main examples of varieties with this property \cite{Koeck}*{Prop.~(1.3)}.
In \Cref{sec:K-of-flags}, we briefly recall the well-known representation-theoretic description of their K-rings.

\subsection{Witt rings as Tate cohomology rings}\label{sec:W=h}
\newcommand{\catdual}{\vee}
Let \( (\cat D,\star,\omega) \) be a triangulated category with duality in the sense of \cite{Balmer:TWGI}.
Then the duality functor \(\star\) induces an involution on the K-group \(\K(\cat D)\), which we still denote \(\star\).
In a slight abuse of terminology (\cf \Cref{rem:tate}), we define the associated Tate cohomology groups as
\begin{align*}
  h^+(\K(\cat{D}), \star ) &:= \frac{\ker(\id -\;\star)}{\im(\id +\;\star)}\\
  h^-(\K(\cat{D}), \star ) &:= \frac{\ker(\id +\;\star)}{\im(\id -\;\star)}
\end{align*}
and write \(h^\total(\K(\cat D),\star)\) for the \(\Z/2\)-graded group \(h^+(\K(\cat D),\star)\oplus h^-(\K(\cat D),\star)\).   Occasionally, the alternative notation
\begin{align*}
  h^i(\K(\cat{D}), \star)
  &:=\begin{cases}
    h^+(\K(\cat{D}),\star) & \text{ if } i\in\Z \text{ is even}\\
    h^-(\K(\cat{D}),\star) & \text{ if } i\in\Z \text{ is odd}
  \end{cases}
\end{align*}
allows more concise statements.  However, independently of which notation is used, we always think of \(h^\total(\K(\cat D),\star)\) as \(\Z/2\)- rather than \(\Z\)-graded.  The few basic facts surrounding these cohomology groups that we will need, like the existence of long exact Tate cohomology sequences, are summarized in \Cref{sec:twisted-tate} below.

Let \( \GW^i(\cat D) \) denote the \(i^\text{th}\)-shifted \GrothendieckWitt group of \( (\cat D, \star, \omega) \) and let \( F^i\colon \GW^i(\cat D)\rightarrow \K(\cat D) \) and \( H^i\colon \K(\cat D)\rightarrow \GW^i(\cat D) \) denote the forgetful and hyperbolic maps, respectively. By \cite{Walter:TGW} they fit into fundamental exact sequences
\begin{equation}\label{eq:Karoubi-exact-sequence}
\GW^{i-1}(\cat D) \xrightarrow{F^{i-1}} \K(\cat D) \xrightarrow{\;H^i\;} \GW^i(\cat D) \rightarrow \W^i(\cat D) \rightarrow 0.
\end{equation}
In particular, \( H^iF^{i-1} = 0 \), and the cokernel \( \GW^i(\cat D)/H^i \) may be identified with the Witt group \( \W^i(\cat D) \). The forgetful and hyperbolic maps interact with the involution on \( \K(\cat{D}) \) as follows:%
\begin{align*}
  H^i\star  &= (-1)^iH^i \\
  F^iH^i    &= \id + (-1)^{i}\star \\
  \star F^i &= (-1)^iF^i
\end{align*}
The following consequence of these equations was observed by Bousfield in the context of real topological K-theory  \cite{Bousfield:2-primary}*{Lemma~4.7}.
\begin{lem}\label{lem:well-defined}
  For any triangulated category with duality \( (\cat{D},\star,\omega) \), the hyperbolic and forgetful maps induce well-defined maps
  \begin{align*}
    \W^0(\cat{D}) \oplus \W^2(\cat{D}) \xrightarrow{(\bar F^0,\bar F^2)} & \;h^+(\K(\cat{D}),\star) \xrightarrow{\mm{\bar H^{-1}\\\bar H^1}} \ker F^{-1} \oplus \ker F^1 \\
    \W^1(\cat{D}) \oplus \W^3(\cat{D}) \xrightarrow{(\bar F^1,\bar F^3)} & \;h^-(\K(\cat{D}),\star) \xrightarrow{\mm{\bar H^0\\\bar H^2}} \ker F^0 \oplus \ker F^2
  \end{align*}
  The two horizontal compositions are of the form \( \mm{H^{i-1}F^i & 0 \\ 0 & H^{i+1}F^{i+2}} \).\qed
\end{lem}
Here is a simple scenario in which these maps are isomorphisms:
\begin{thm}\label{thm:W=h-abelian}
  Let \(\cat D\) be the derived category of a \(k\)-linear abelian category with an exact \(k\)-linear duality \((\cat A,\dual,\omega)\), where \(k\) is an algebraically closed field of characteristic not two.  Suppose that all objects of \(\cat A\) have finite length.
  Then the maps \((\bar F^i, \bar F^{i+2})\) and \(\mm{\bar H^{i}\\\bar H^{i+2}}\) of \Cref{lem:well-defined} are isomorphisms.\footnote{Note that \(h^-(\K(\cat{D}),\star)\), \(\W^1(\cat{D})\) and \(\W^3(\cat{D})\) vanish in this setting \cite{BalmerWalter:GWSS}*{Prop.~5.2}.}
\end{thm}
This theorem applies for instance to the category of finite-dimensional linear representations of an affine algebraic group over an algebraically closed field.  Strictly speaking, we will not use this result in the sequel. We include it here partly because it puts \Cref{thm:W=h} below into perspective, and partly because it has a simple, illustrative proof.

Now let \(X\) be variety over a field of characteristic not two.
For any line bundle \(\lb L\) over \(X\), we can consider the category with duality \((\cat D(X),\star_{\lb L},\omega)\), the derived category of the category of vector bundles over \(X\) equipped with the duality \(\star_{\lb L}\) induce by \(\vb E\mapsto \cat Hom(\vb E,\lb L)\) and the usual double-dual identification \(\omega\).
Then \(\K(\cat D(X))\) is the usual K-group \(\K_0(X)\) of \(X\), equipped with an induced involution \(\star_{\lb L}\) for each line bundle \(\lb L\) over \(X\).
To simplify notation, we define
\begin{alignat*}{5}
  &h^\total(X)          &&:= h^\total(\K_0(X),\star_{\lb O_X})\\
  &h^\total(X,\lb{L})   &&:= h^\total(\K_0(X),\star_{\lb L})
\end{alignat*}
For any two line bundles \(\lb L_1\) and \(\lb L_2\), the ring structure on \(\K_0(X)\) induces a pairing
\[
h^i(X,\lb L_1) \otimes h^j(X,\lb L_2) \to h^{i+j}(X,\lb L_1\otimes L_2)
\]
such that \(h^\total(X)\) is a \(\Z/2\)-graded ring, and such that \(h^\total(X,\lb{L})\) is a graded \(h^\total(X)\)-module (\cf \Cref{tate-rings-and-modules}).

\begin{thm}\label{thm:W=h}
  Let \(X\) be a smooth cellular variety over an algebraically closed field of characteristic not two,
  and let \(\lb L\) be a line bundle over \(X\). Then for \((\mathcal D,\star, \omega):=(\mathcal D(X),\star_{\lb L},\omega)\) as above,
  the maps in Lemma~\ref{lem:well-defined} are isomorphisms.
  In particular, we have isomorphisms\footnote{%
    Recall that we write \(\W^{i,i+2}(X,\lb L)\) for \(\W^i(X,\lb L)\oplus\W^{i+2}(X,\lb L)\).}
  \begin{align*}
    \W^{0,2}(X,\lb L) &\cong h^+(X,\lb L)\\
    \W^{1,3}(X,\lb L) &\cong h^-(X,\lb L)
  \end{align*}
  When \(\lb L\) is trivial, these group isomorphisms assemble to an isomorphism of \(\Z/2\)-graded rings \( \W^\total(X) \cong h^\total(X)\).
  More generally, for any two line bundles \(\lb L_1\) and \(\lb L_2\) over \(X\), these isomorphisms are compatible with the respective pairings in the sense that
  \[\xymatrix{
    \W^{i,i+2}(X,\lb L_1)\otimes \W^{j,j+2}(X,\lb L_2) \ar[r] \ar[d] &  \W^{i+j,i+j+2}(X,\lb L_1\otimes\lb L_2)\ar[d] \\
    h^i(X,\lb L_1)\otimes h^j(X,\lb L_2) \ar[r]     &  h^{i+j}(X,\lb L_1\otimes \lb L_2)
  }\]
  commutes.
\end{thm}

The remainder of this section is occupied by the proofs of \Cref{thm:W=h-abelian,thm:W=h}.

\begin{proof}[Proof of \Cref{thm:W=h-abelian}]
  As two is assumed to be invertible, we can identify \(\W^0(\cat D)\) and \(\W^2(\cat D)\) with the usual Witt groups of symmetric and anti-symmetric objects \(\W^0(\cat A)\) and \(\W^2(\cat A)\) \cite{Balmer:TWGII}*{Theorem~4.3}, and likewise for the \GrothendieckWitt groups \(\GW^0(\cat D)\) and \(\GW^2(\cat D)\). Moreover, by Proposition~5.2 of \cite{BalmerWalter:GWSS}, \(\W^1(\cat D) = \W^3(\cat D) = 0\).

  \newcommand{\myset}{\mathfrak S}
  Choose a set \(\myset\) of representatives of the simple objects of \(\cat A\).
  As \(k\) is algebraically closed, every such object has \(k\) as endomorphism ring.  Therefore, every self-dual simple object is either symmetric or anti-symmetric, exclusively \citelist{\cite{QSS}*{Prop.~2.5 (1)}\cite{CalmesHornbostel:reductive}*{Lemma~1.21}}.  Let \(\myset_+, \myset_- \subset \myset\) be the subsets of symmetric and anti-symmetric objects, and let \(\myset_0\) contain one object of each pair \((S,S^\dual)\) of non-self-dual objects of \(\myset\).
Given any object \(A\) of \(\cat A\), let \(\cat A_A\) denote the full subcategory on objects isomorphic to direct summands of finite direct sums of copies of \(A\).  The inclusions of \(\cat A_S\) and \(\cat A_{S\oplus S^\dual}\) into \(\cat A\) induce isomorphisms
\begin{align*}
  \K(\cat A) &\cong \bigoplus_{S\in\myset_+} \K(\cat A_S) \oplus \bigoplus_{S\in\myset_-} \K(\cat A_S) \oplus \bigoplus_{S\in\myset_0} \K(\cat A_{S\oplus S^\dual}).
\end{align*}
Moreover, as the duality on \(\cat A\) restricts to \(\cat A_S\) (for self-dual \(S\)) and to \(\cat A_{S\oplus S^\dual}\) (for arbitrary \(S\)), we have analogous decompositions for \(\GW^0(\cat A)\), \(\GW^2(\cat A)\), \(\W^0(\cat A)\) and \(\W^2(\cat A)\) \citelist{\cite{QSS}*{Theorem~6.10}\cite{CalmesHornbostel:reductive}*{Cor.~1.14}}.
Moreover, these decompositions are compatible with the hyperbolic and forgetful maps \(F^i\) and \(H^i\).
It follows that we have such decompositions for all \(\GW^i(\cat D)\) and \(\W^i(\cat D)\), with each category \(\cat A_A\) replaced by its derived category \(\cat D_A\).
Thus, it suffices to prove the \namecref{thm:W=h-abelian} for each category \(\cat A_S\) with \(S\in \myset_+\) or \(\myset_-\) and each category \(\cat A_{S\oplus S^\dual}\) with \(S\in\myset_0\).

For \(S\in\myset_+\), we have \(\K(\cat A_S) \cong \K(k)\), \(\GW^0(\cat A_S) \cong \GW^0(k)\) and \(\GW^2(\cat A_S) \cong \GW^2(k)\) \cite{QSS}*{Prop.~2.4}.\footnote{%
We use the common shorthand \(\K(k)\) for the K-group of the category of finite-dimensional \(k\)-vector spaces, or, phrased geometrically, for the K-group of \(\mathrm{Spec}(k)\).  The same convention is used for Witt and \GrothendieckWitt groups.}
The groups \(\K(k)\) and \(\GW^2(k)\) are equal to \(\Z\) for any field \(k\).  As \(k\) is algebraically closed, the same is true for \(\GW^0(k)\).  The proof of Proposition~2.4 in \cite{QSS} also yields explicit generators for each of these groups:  if we choose some fixed symmetric form \(\sigma\) on \(S\), we can write:
\begin{align*}
  \K(\cat A_S) &= \Z\cdot [S]\\
  \GW^0(\cat A_S) &= \Z\cdot [S,\sigma]\\
  \GW^2(\cat A_S) &= \Z\cdot [H^2S]
\end{align*}
The exact sequences \eqref{eq:Karoubi-exact-sequence} therefore take the following form:
\begin{alignat*}{11}
  \Z/2&[H^{-1}S] &&\xrightarrow{F^{-1}} &\Z&[S] \xrightarrow{H^0} &\Z&[S,\sigma] &&\to &\Z/2&[S,\sigma] \to &&\quad 0\\
  \Z&[S,\sigma] &&\xrightarrow{F^0} &\Z&[S] \xrightarrow{H^1} && 0            &&\to &0&                    && \\
  &0            &&\xrightarrow{F^1}&\Z&[S] \xrightarrow{H^2} &\Z&[H^2S]   &&\to &0&                      && \\
  \Z&[H^2S]   &&\xrightarrow{F^2} &\Z&[S] \xrightarrow{H^{-1}} &\Z/2&[H^{-1}S]  &&\to &0&
\end{alignat*}
As the involution on \(\Z[S]\) is trivial, \(h^+(\Z[S])= \Z/2[S]\), and the claims that the maps
\((\bar F^i, \bar F^{i+2})\) and \(\mm{\bar H^{i}\\\bar H^{i+2}}\)
are isomorphisms are easily checked.

The proof for \(S\in\myset_-\) is analogous.  The involution on \(\K(\cat A_S)\) is again trivial, and the exact sequences \eqref{eq:Karoubi-exact-sequence} take the following form, where \(\sigma\) is some fixed anti-symmetric form on \(S\):
\begin{alignat*}{11}
  &0            &&\xrightarrow{F^{-1}}&\Z&[S] \xrightarrow{H^0} &\Z&[H^0S]   &&\to &0&                      && \\
  \Z&[H^0S]   &&\xrightarrow{F^0} &\Z&[S] \xrightarrow{H^1} &\Z/2&[H^1S]  &&\to &0&                    \\
  \Z/2&[H^1S] &&\xrightarrow{F^1} &\Z&[S] \xrightarrow{H^2} &\Z&[S,\sigma] &&\to &\Z/2&[S,\sigma] \to && \quad 0\\
  \Z&[S,\sigma] &&\xrightarrow{F^2} &\Z&[S] \xrightarrow{H^{-1}} && 0            &&\to &0&                    &&
\end{alignat*}

Finally, for \(\cat A_{S\oplus S^\dual}\) with \(S\in\myset_0\), all groups involved in \Cref{lem:well-defined} vanish.  In slightly more detail, the \GrothendieckWitt groups of \(\cat A_{S\oplus S^\dual}\) can be identified with the \GrothendieckWitt groups of the ring with involution \((k\times k, \circ)\), where \(\circ\) interchanges the two factors.  Any finitely-generated module over \(k\times k\) can be written as \(V_1\oplus V_2\) for \(k\)-vector spaces \(V_1\) and \(V_2\), with the first factor of \(k\times k\) acting trivially on \(V_2\) and the second factor acting trivially on \(V_1\).
In this notation, the dual of \(V_1\oplus V_2\) is \(V_2^\dual\oplus V_1^\dual\), where \((-)^\dual\) denotes the usual duality on vector spaces.  It follows that any (anti-)symmetric module over \((k\times k,\circ)\) is necessarily hyperbolic \cite{Karoubi:fondamental}*{\S~1.2}.
Thus, in this case the exact sequences \eqref{eq:Karoubi-exact-sequence} take the following form:
\begin{alignat*}{11}
  \Z&[H^{-1}S] &&\xhookrightarrow{F^{-1}}&\Z&[S]\oplus\Z[S^\dual]\xrightarrow{H^0}   &\Z&[H^0S]     &&\to\quad&0 \\
  \Z&[H^0S]    &&\xhookrightarrow{F^0}   &\Z&[S]\oplus\Z[S^\dual]\xrightarrow{H^1}   &\Z&[H^1S]     &&\to &0 \\
  \Z&[H^1S]    &&\xhookrightarrow{F^1}   &\Z&[S]\oplus\Z[S^\dual]\xrightarrow{H^2}   &\Z&[H^2S]     &&\to &0 \\
  \Z&[H^2S]    &&\xhookrightarrow{F^2}   &\Z&[S]\oplus\Z[S^\dual]\xrightarrow{H^{-1}} &\Z&[H^{-1}S] &&\to &0
\end{alignat*}
In particular, all the maps \(F^i\) are injective and all Witt groups are zero, in agreement with the fact that \(h^\total(\Z[S]\oplus\Z[S^\dual])\) vanishes.
\end{proof}

For the proof of \Cref{thm:W=h}, we use a couple of lemmas.
\begin{lem}\label{lem:boundary}
  Let \( \cat{A} \xrightarrow{i} \cat{B} \xrightarrow{r} \cat{C} \) be a short exact sequence of triangulated categories with duality with the property that
  \[
  0 \rightarrow \K(\cat{A}) \rightarrow \K(\cat{B}) \rightarrow \K(\cat{C}) \rightarrow 0
  \]
  is exact.
  Then the forgetful functors \( F^i \) commute with the boundary maps in the associated long exact sequence of Witt groups and the long exact Tate cohomology sequence, \ie the following square commutes:
  \[
  \xymatrix{
    {\W^i(\cat{C})} \ar[r]^{\partial}\ar[d]^{F^i} & \W^{i+1}(\cat{A})\ar[d]^{F^{i+1}} \\
    {h^i(\cat{C})}  \ar[r]^{\partial}            & {h^{i+1}(\cat{A})}
  }
  \]
\end{lem}
\begin{proof}
  The boundary map on Witt groups is defined as follows \cite{Balmer:TWGI}*{Def.~4.16}:\footnote{C.f. \cite{BG:Koszul}*{\S~2} for a concise summary in the case of vector bundles over varieties.}
  Any element of \( \W^i(\cat{C}) \) can be represented by the image \(r(B,\beta)\) of some \(i\)-symmetric pair \( (B,\beta) \) in \( \cat{B} \).
  Then the cone of \(\beta\colon B\to B^\dual[i]\) is isomorphic to zero in \( \cat{C} \), so it lies in \(\cat{A}\). It may be equipped with some \( (i+1) \)-symmetric form \( \gamma \), and we have \( \partial[r(B,\beta)] := [\cone(\beta), \gamma] \). In particular, \( F\partial[r(B,\beta)] = [\cone(\beta)] \).

  The boundary map from \( h^i(\cat{C}) \) to \( h^{i+1}(\cat{A}) \), on the other hand, sends \( [rB] \) to \( [B] - (-1)^{i}[B^\dual]\), which we may view as an element of \( \K(\cat{A}) \) (\cf \Cref{tate-exact-sequence}). This element agrees with \( [\cone(\beta)] \) in \( \K(\cat{B}) \), so we conclude using the injectivity of \mbox{\( \K(\cat{A}) \rightarrow \K(\cat{B}) \)}.
\end{proof}

\begin{lem}\label{lem:Thom}
  Let \( Z \hookrightarrow X \) be a closed embedding of smooth quasi-projective varieties over a field of characteristic not two, and let \(\lb L\) be a line bundle over \(X\).  The forgetful functor commutes with the Thom isomorphism of \cite{Nenashev:Gysin}, \ie we have commutative diagrams of the form
  \[
  \xymatrix{
    {\W^{i-c}(Z,\lb{L}|_Z\otimes\det{\vb{N}})} \ar[r]_-{\cong}\ar[d]^{F^{i-c}} & {\W^i_Z(X,\lb{L})} \ar[d]^{F^i} \\
    {h^{i-c}(Z,\lb{L}|_Z\otimes\det{\vb{N}})} \ar[r]_-{\cong}                & {h^i_Z(X,\lb{L})}
  }
  \]
  Here, \( c \) and \(\vb N\) are the codimension and the normal bundle of  \(Z\hookrightarrow X\).  We have written \(h^i_Z(X,\lb L)\) for the Tate cohomology  \(h^i(\K^Z_0(X),\star_{\lb L})\).
\end{lem}
\begin{proof}
  A review of the construction of the Thom isomorphism shows that the claim already holds on the level of \GrothendieckWitt and K-groups:
  Let \( K(\vb{N}) \) denote the Koszul complex of \( \vb{N} \) as described in \cite{Nenashev:Gysin}*{\S~2}.  This is a complex of vector bundles over the affine bundle associated with \(\vb{N}\),  with cohomology supported on the zero section \( Z\). It carries a canonical \(c\)-symmetric form
  \[
  \Theta(\vb{N})\colon K(\vb{N})\xrightarrow{\simeq} K(\vb{N})^\dual[c]\otimes p^*\det{\vb{N}^\dual},
  \]
  where \( p\) denotes the bundle projection \(\vb{N}\twoheadrightarrow Z\).
  One thus has elements \mbox{\( \kappa := [K(\vb{N}),\Theta(\vb{N})] \)} in \( \GW^c_Z(\vb{N},\det{p^*\vb{N}^\dual}) \) and \( F^c(\kappa) \) in \( K_0^Z(\vb{N}) \). Multiplication with \( \kappa \) induces group isomorphisms
  \[
  \GW^{i-c}(Z,\lb{L}\otimes\det{\vb{N}}) \xrightarrow{\quad\cong\quad} \GW^i_Z(\vb{N},p^*\lb{L})
  \]
  for all line bundles \( \lb{L} \) over \( Z\), while multiplication with \( F^c(\kappa) \) induces the usual Thom-isomorphism of K-groups \( \K_0(Z) \cong \K_0^Z(X) \).
  If we want to view this as an isomorphism of groups with involution, we need to take care both of the sign and the twist of the involution. So let \( \K_0(X,-\lb{L}) \) denote \( \K_0(X) \) equipped with minus the \( \lb{L}\)-twisted duality (for any \( X \) and any \( \lb{L} \)). Then multiplication by \( F^c(\kappa) \) induces isomorphisms
  \[
  \K_0(Z,(-1)^{i-c}\lb{L}\otimes\det{\vb{N}}) \xrightarrow{\quad\cong\quad} \K_0^Z(\vb{N},(-1)^ip^*\lb{L})
  \]
  and, since \( F \) is multiplicative, we have commutative diagrams
  \[
  \xymatrix{
    {\GW^{i-c}(Z,\lb{L}\otimes\det{\vb{N}})}\ar[r]^-{\cong}_-{\cdot\kappa}\ar[d]^{F^{i-c}} & {\GW^i_Z(\vb{N},p^*\lb{L})} \ar[d]^{F^i} \\
    {\K_0(Z,(-1)^{i-c}\lb{L}\otimes\det{\vb{N}})} \ar[r]^-{\cong}_-{\cdot F^c{\kappa}} & {\K_0^Z(\vb{N},(-1)^ip^*\lb{L})}
  }
  \]
  The Thom-isomorphisms in the lemma are obtained by composing the horizontal arrows with \( p^* \). Thus, the lemma follows from the naturality of \( F\) and the observation that \( h^j(\K_0(X,(-1)^i\lb{L})) = h^{j+i}(X,\lb{L}) \).
\end{proof}

\begin{proof}[Proof of Theorem~\ref{thm:W=h}]
  By Lemma~\ref{lem:Fiso-Hiso} below, it suffices to show that the maps \((\bar F^i, \bar F^{i+2})\) are isomorphisms. This may be proved by induction over the cells of \( X \). Indeed, using \Cref{thm:W=h-abelian} or otherwise, we see that the claim is true for \( X = \mathrm{Spec}(k) \), where \( k \) is an algebraically closed field.  Thus, by homotopy invariance, the claim holds for an arbitrary cell \( \A^n_k \). Now  let \( X = U_0 \supset U_1 \supset U_2 \cdots \supset U_N=\emptyset \) be a filtration of \( X \) by open subsets as in \cite{Me:WCCV}*{proof of Thm~2.6}, \ie such that the complement of each \( U_{k+1} \) in \( U_k \) is a closed cell \( Z_k \).  For each \( k \) we have a short exact sequence of triangulated categories
  \begin{equation}
  0 \rightarrow \cat{D}^b_{Z_k}(U_k) \rightarrow \cat{D}^b(U_k) \rightarrow \cat{D}^b(U_{k+1}) \rightarrow 0.
  \end{equation}
  We claim that the first map in the associated sequence of K-groups is injective, \ie that
 \begin{equation}\label{eq:K-ses}
  0 \rightarrow \K_0^{Z_k}(U_k) \rightarrow \K_0(U_k) \rightarrow \K_0(U_{k+1}) \rightarrow 0
\end{equation}
is still exact.  Indeed, note that all three K-groups are free abelian and of finite rank.  The first, \(\K_0^{Z_k}(U_k)\), can be identified via dévissage with the K-group of the cell \(Z_k\), so it is free abelian of rank one.  The other two are the K-groups of the smooth cellular varieties \(U_k\) and \(U_{k+1}\), and the K-group of such a variety is free abelian of rank equal to the number of cells.\footnote{This well-known fact may be checked, for example, by passing to coherent K-groups and noting that the localization sequences associated with the canonical cellular filtration by closed subsets split.  With slightly more effort than we care to take here, one may also deduce that the above localization sequences of K-groups associated with the open filtration split in all degrees, not just in degree zero.}  Thus, in the above sequence, the rank of the central group is equal to the sum of the ranks of the other two groups, and the claim follows.

  If we choose the usual dualities on the above categories (or if we twist the dualities by the restriction of \( \lb{L} \) in each case) this sequence becomes a short exact sequence of triangulated categories with duality. By Lemma~\ref{lem:boundary}, the associated long exact sequence of Witt groups may be compared to the long exact Tate cohomology sequence induced by \eqref{eq:K-ses} via a commutative ladder diagram:
\begin{adjustwidth}{-5cm}{-5cm}
  \[
  \xymatrix@C=7pt{
    {\cdots} \ar[r]
    & {\W^{i-1,i+1}(U_{k+1})}   \ar[r]^-{\partial}\ar[d]
    & {\W^{i,i+2}_{Z_k}(U_k)}  \ar[r]\ar[d]
    & {\W^{i,i+2}(U_k)}           \ar[r]\ar[d]
    & {\W^{i,i+2}(U_{k+1})}     \ar[r]^-{\partial}\ar[d]
    & {\W^{i+1,i+3}_{Z_k}(U_k)} \ar[r]\ar[d]
    & {\cdots} \\
    {\cdots} \ar[r]
    & {h^{i-1}(U_{k+1})} \ar[r]^-{\partial}
    & {h^i_{Z_k}(U_k)} \ar[r]
    & {h^i(U_k)} \ar[r]
    & {h^i(U_{k+1})} \ar[r]^-{\partial}
    & {h^{i+1}_{Z_k}(U_k)} \ar[r]
    & {\cdots}
  }
  \]
\end{adjustwidth}
  By (backward) induction, we may assume that the vertical maps to \( h^\total(U_{k+1}) \) are isomorphisms. By Lemma~\ref{lem:Thom} the maps to \( h^\total_{Z_k}(U_k) \) may be identified with the maps
  \[
  \W^{i-c,i-c+2}(Z_k) \longrightarrow h^{i-c}(Z_k),
  \]
  which are isomorphisms since \( Z_k\cong\A^n\) for some \(n\). So we may conclude that the vertical maps to \( h^\total(U_k) \) are also isomorphisms.

Finally, the claim concerning the multiplicative structure follows from the corresponding claim for the forgetful functors, \ie from the fact that
\[\xymatrix{
  \GW^i(X,\lb L_1)\otimes \GW^j(X,\lb L_2) \ar[r]^-{\cdot} \ar[d]^{F^i\otimes F^j}&  \GW^{i+j}(X,\lb L_1\otimes\lb L_2)\ar[d]^{F^{i+j}} \\
  \K_0(X)\otimes \K_0(X) \ar[r]^-{\cdot}     &  \K_0(X)
}\]
commutes.
\end{proof}

\begin{lem}\label{lem:Fiso-Hiso}
  Let \( (\cat D,\star,\omega) \) be as in \Cref{lem:well-defined}.
  If \( (\bar F^i,\bar F^{i+2}) \) is an isomorphism for all (\ie for both even and odd) values of \( i \), then so is \( \mm{\bar H^i\\\bar H^{i+2}} \).
\end{lem}
\begin{proof}
  To see that \( \mm{\bar H^0\\\bar H^2} \) is injective, suppose that \( H^0(x) = H^2(x) = 0 \) for some element \( x \in \K(\cat{D}) \) which is  anti-self-dual (\ie satisfies \(x^\star = -x\)) .  Then by the exactness of sequence~\eqref{eq:Karoubi-exact-sequence} there exist elements \( y_{-1}\in\GW^{-1}(\cat{D}) \) and \( y_1 \in\GW^1(\cat{D}) \) such that \( x = F^{-1}(y_{-1}) = F^1(y_1) \). Injectivity of \( (\bar F^{-1}, \bar F^1) \) then implies that \( (y_{-1},-y_1) \) is zero in \(\W^{-1}(\cat{D})\oplus\W^1(\cat{D})\). By definition of the Witt group, this means that we may find elements \( z_{-1},z_1 \in \K(\cat{D}) \) such that \(y_{-1}=H^{-1}(z_{-1}) \) and \( y_1=H^1(z_1) \). In particular, \( x = F^1(H^1(z_1)) = z_1 - z_1^*\). So \( x \) is zero in \( h^-(\cat{D}) \).

  In order to see that \( \bar H^0 \) surjects onto \( \ker F^0 \), take an arbitrary element \( x \in \GW^0(\cat D) \) such that \( F^0(x) = 0\). Injectivity of \( (\bar F^0, \bar F^2) \) then implies that \( x \) is zero in \( \W^0(\cat D) \). So \( x = H^0(z) \) for some \( z \in \K(\cat{D}) \). Since \( z + z^* = F^0(x) = 0 \), we may consider the class of \( z \) in \( h^-(\cat{D}) \). Surjectivity of \( (\bar F^{-1}, \bar F^1) \) then implies the existence of elements \( y_{-1}\in \W^{-1}(\cat D) \) and \( y_1\in \W^1(\cat D) \) such that \( z = \bar F^{-1}(y_{-1}) + \bar F^1(y_1) \). Applying \( \bar H^0\), we find that \( x = \bar H^0 \bar F^1(y_1) \). Thus, \( (x,0) = (\bar H^0, \bar H^2)(F^1(y_1)) \). Similarly, \( \bar H^2 \) surjects onto \( \ker F^2 \), so \( \mm{\bar H^0\\\bar H^2} \) is surjective.

  The proof that \( \mm{\bar H^{-1}\\\bar H^1} \) is an isomorphism is analogous.
\end{proof}

\subsection{K-theory of flag varieties}\label{sec:K-of-flags}
\Cref{thm:W=h} tells us that the Witt groups of a flag variety \(G/P\) are closely related to the Tate cohomology of its algebraic K-group \(\K_0(G/P)\). The reason why this is useful is that this K-group has a nice description in terms of the representation rings \(\Rep(G)\) and \(\Rep(P)\) of \(G\) and \(P\). More precisely, Panin \cite{Panin:TwistedFlags} has proved the following algebraic variant of a theorem of Hodgkin \cite{Hodgkin}*{Lemma~9.2}:
\begin{thmPanin}\label{thm:Hodgkin}
Let \(G\) be a simply connected semisimple algebraic group with a parabolic subgroup \(P\).
Then we have a ring isomorphism
\[
\K_0(G/P) \cong \factor{\Rep(P)}{\ideal a}
\]
where \(\ideal a\subset \Rep(P)\) is the ideal generated by restrictions of rank zero classes in \(\Rep(G)\).
\end{thmPanin}
This isomorphism is very geometric:
It is induced by the morphism \(\K_0(G/P)\leftarrow\Rep(P)\) that sends a representation \(V\) of \(P\) to the associated vector bundle \(G\times_P V\) over \(G/P\) (\cf \cite{Panin:TwistedFlags}*{\S~1}).
In particular, it respects the involutions on both sides induced by the usual duality on vector bundles and representations,
and it restricts to the isomorphism
\begin{equation}\label{eq:Pic=X}
\Pic(G/P) \cong X^*(\Rep(P))
\end{equation}
\cite{MerkurjevTignol}*{Prop.~2.3} under which the line bundle called \(\lb L_\beta\) in \Cref{sec:vanishing} corresponds to the character \(\omega_\beta\) of \(\Rep(P)\) (see \Cref{sec:rep-of-P}).

A key ingredient in Panin's proof is the following theorem of Steinberg \cite{Steinberg:Pittie}, which we will also need later:
\begin{thmSteinberg}\label{thm:Steinberg}
  The representation ring \(\Rep(P)\) of a parabolic subgroup of a simply connected semisimple algebraic group \(G\) is finite and free as a \(\Rep(G)\)-module.
\end{thmSteinberg}

The remainder of the article is now easily outlined:
we work out an explicit description of \(\Rep(P)\) as a ring with involution,
we compute its Tate cohomology, and finally we translate everything back to Witt groups.

\encouragepagebreak
\section{Representation rings}
In this section, we analyse the representation ring of a parabolic subgroup and the involution induced by dualizing representations.

To put the final result in \Cref{sec:involution-on-RepP} into perspective, recall that the representation ring of any simply connected semisimple algebraic group can be written as a polynomial ring.  We may even choose a set of polynomial generators which is preserved by the involution.
The representation ring \(\Rep(P)\) of a parabolic subgroup is only slightly more complicated as a ring:  it may be decomposed into a tensor product of a polynomial ring and a ring of Laurent polynomials. However, in general the involution does not respect this decomposition:  while the involution does restrict to an involution of the Laurent part (on which it is given by multiplicative inversion), the dual of a polynomial generator of \(\Rep(P)\) is generally a product of a `polynomial factor' and a non-trivial `Laurent factor'.

Throughout this section, the ground field may be arbitrary.

\subsection{Representation rings of reductive groups}\label{rep-of-reductive}
Before specializing to (Levi subgroups of) parabolic groups, we recall the general description of the representation ring of a split reductive group. So let \(G\) be a split reductive group over a field, with split maximal torus \(T\). We use the following notation:
\begin{quote}
\(X^* := X^*(T)\) the group of characters of \(T\) \\
\(X_* := X_*(T)\) the group of cocharacters of \(T\) \\
\(\pairing{-}{-}\) the canonical pairing \(X^*\times X_* \to \Z\)

\(R\;\; \subset X^*\) the set of roots\\
\(R^\vee \subset X_*\) the set of coroots

\(R^+\subset R\;\;\) a choice of positive roots\\
\(\simpleR\;\; \subset R^+\) the corresponding set of simple roots

\(\Weyl\) the Weyl group
\end{quote}
\begin{thmSerre}[\cite{Serre}*{Théorème~4}]\label{character-iso}
  For any split reductive group \(G\) over a field, the character isomorphism \(\Rep(T)\to\Z[X^*]\) restricts to an isomorphism
  \[
  \Rep(G) \xrightarrow{\cong} \Z[X^*]^{\Weyl}.
  \]
\end{thmSerre}
Additively, the invariants under the Weyl group can be described as follows. Let
\[
\FWC := \left\{ \xi\in X^* \mid \pairing{\xi}{\sigma^\vee} \geq 0 \text{ for all } \sigma\in\simpleR \right\}
\]
be the semigroup of dominant weights corresponding to our choice of simple roots \(\simpleR\). It is a fundamental domain for the action of the Weyl group. Therefore
\[
   \Z[X^*]^{\Weyl} = \bigoplus_{\delta\in\FWC}\Z\cdot S(\delta),
\]
where \(S(\delta)\) is the sum over the elements of the orbit of \(\delta\) under the Weyl group: \(S(\delta) = \sum_{\xi\in\Weyl.\delta}e^{\xi}\).

As for the multiplicative structure, it is well known that when \(G\) is semisimple and simply connected, \(\Rep(G)\) is a polynomial ring. More generally, if \(G\) is reductive with simply connected derived group \(\mathcal DG\), the representation ring of \(G\) may be identified with a product of the polynomial ring \(\Rep(\mathcal DG)\) and the ring of Laurent polynomials on the characters of \(G\) \citelist{\cite{SGA6}*{Exposé~0 App: RRR}\cite{Hodgkin}*{Prop.~11.1 (p.~81)}}. However, this identification depends on certain choices. We provide a proof that will suit our later needs.

We decompose \(X^*\) into a `toral part' and a `semisimple part'. For the first part, we take the character group of \(G\):
\begin{align*}
X^*(G) = (\simpleR^\vee)^\perp := \left\{\xi\in X^* \mid \pairing{\xi}{\sigma^\vee} = 0 \text{ for all } \sigma\in \simpleR\right\}
\end{align*}
The quotient \(\factor{X^*}{(\simpleR^\vee)^\perp}\) can be identified with the character group of a maximal torus \(T_{\mathcal DG}\) of the derived group \(\mathcal DG\)
\cite{Jantzen}*{II.1.18}.
Let \(\Omega\) be an arbitrary\footnote{%
We will see that when dealing with parabolic subgroups of simply connected semisimple groups, there is a natural choice for \(\Omega\). See also \Cref{choice-of-complement}.}
complement of \((\simpleR^\vee)^\perp\) in \(X^*\). Then the inclusion  \(i\colon T_{\mathcal DG}\hookrightarrow T\) induces an isomorphism \(i^*\) from \(\Omega\) to \(X^*(T_{\mathcal DG}) =: X^*_{\mathcal DG}\).  We thus have decompositions:
\[
\begin{aligned}
  X^*
  &= (\simpleR^\vee)^\perp \oplus \Omega \\
  &\cong X^*(G) \oplus X^*_{\mathcal DG}
\end{aligned}
\quad\quad\quad
\begin{aligned}
  \FWC
  &= (\simpleR^\vee)^\perp \oplus  (\FWC \cap\Omega) \\
  &\cong X^*(G) \oplus \FWC_{\mathcal DG}
\end{aligned}
\]
When \(\mathcal DG\) is simply connected, its set of dominant weights \({\FWC_{\mathcal DG}}\) is a \emph{free} abelian semigroup, so the isomorphism above shows that \(\FWC\) decomposes as a direct sum of a free abelian group \((\simpleR^\vee)^\perp\) with trivial \(\Weyl\)-action and a free abelian semigroup \(\FWC\cap\Omega\).
\begin{thm}\label{reps-of-reductive}
  Let \(G\) be a split reductive group whose derived group \(\mathcal DG\) is simply connected.
  Let \(X^*= X^*(G)\oplus \Omega\) be a decomposition as above. Choose a basis \(\xi_1,\dots, \xi_l\) of \(X^*(G)\) and a basis \(\omega_1,\dots, \omega_r\) of the free semigroup \(\FWC\cap\Omega\). Then we have a ring isomorphism
  \begin{align*}
    \Z[x_1^{\pm 1},\dots,x_l^{\pm l}]\otimes \Z[w_1,\dots,w_r] \xrightarrow{\quad\cong\quad} \Z[X^*]^W \cong \Rep (G)
  \end{align*}
  sending \(x_i\) to \(S(\xi_i) = e^{\xi_i}\) and \(w_i\) to \(S(\omega_i) = \sum_{\xi\in \Weyl.\omega_i} e^\xi\).
\end{thm}
\begin{proof}
Recall that there exists a partial order on the dominant weights \(\FWC\) with the following properties \cite{Adams:Lie}*{6.27, 6.36}:
\begin{compactitem}
  \item Only finitely many dominant weights are smaller than a fixed dominant weight.
  \item For arbitrary dominant weights \(\delta\) and \(\delta'\),
    \[ S(\delta)S(\delta') = S(\delta+\delta') + (\text{smaller terms}).\]
    Here, `\(S(\delta) + (\text{smaller terms})\)' is to be read as a short-hand for:
    there exist coefficients \(a_{\epsilon}\in\Z\) such that
    \[
    P = S(\delta) + \sum_{\mathclap{\epsilon\in \FWC\colon\epsilon\prec\delta}}a_{\epsilon} S(\epsilon).
    \]
\end{compactitem}
Let \(\Phi\) denote the ring homomorphism defined in the theorem.

For surjectivity, it suffices to show that, for any weight \(\delta\) in \(\FWC\), the symmetric sum \( S(\delta) \) has a preimage under \(\Phi\). This may be checked by induction over the number of weights smaller than \(\delta\). Indeed, let us write \(\delta = \sum_i a_i\xi_i + \sum_j b_j\omega_j\) with coefficients \(a_i\in\Z\) and \(b_j\in\N\). By the second property of the ordering above,
  \[
  \Phi(\prod x_i^{a_i}\otimes \prod w_j^{b_j}) = S(\delta) + (\text{smaller terms}).
  \]
  By the inductive hypothesis, we may assume that all terms on the right-hand side apart from \(S(\delta)\) are contained in the image of \(\Phi\), so \(S(\delta)\) itself must be contained therein.

  For injectivity, take an arbitrary element
  \[
  P = \sum_{\quad\quad\mathclap{\substack{\vec a = (a_1,\dots,a_l)\in\Z^l\\\vec b = (b_1,\dots,b_r)\in\N^r}}\quad\quad} c_{\vec a,\vec b} x_1^{a_1}\cdots x_l^{a_l}\cdot w_1^{b_1}\cdots w_r^{b_r}
  \]
  and suppose that \(\Phi(P) = 0\). We can write \(\Phi(P)\) as
  \[
  \Phi(P) = \sum_{\vec a,\vec b} c_{\vec a,\vec b} \left[S(\xi_{\vec a,\vec b}) + (\text{smaller terms})\right]
  \]
  where \(\xi_{\vec a,\vec b} := \sum_i a_i\xi_i + \sum_j b_j \omega_j\).
  The assumption that the \(\omega_j\) form a basis and not just a generating set of \(\FWC\cap \Omega\) ensures that the \(\xi_{\vec a,\vec b}\) and hence the \(S(\xi_{\vec a,\vec b})\) are all distinct.  Let \(\mathfrak S\) be the set of all pairs \((\vec a,\vec b) \) for which \(c_{\vec a,\vec b} \) is non-zero. If \(\mathfrak S\) is non-empty, we may choose some \((\vec a,\vec b)\in \mathfrak S\) such that \(\xi_{\vec a,\vec b}\) is maximal in the sense that it is not smaller than any other \(\xi_{\vec a,\vec b}\) with \((\vec a,\vec b)\in \mathfrak S\).
  Then \(\Phi(P)= 0\) implies that \(c_{\vec a,\vec b} = 0\)---a contradiction. Thus, \(\mathfrak S\) must be empty and we find that \( P=0\).
\end{proof}

\begin{rem}[Decomposition of \(X^*_\Q\)]\label{choice-of-complement}
If we pass to \(X^*_\Q = X^*\otimes_\Z\Q\), we have a canonical decomposition
\[
X^*_\Q = (\simpleR^\vee)^\perp_\Q \oplus \Q \simpleR.
\]
However, \((\simpleR^\vee)^\perp\oplus (\Q \simpleR\cap X^*)\)
is in general only a finite index subgroup of \(X^*\).
Geometrically, \(\Q \simpleR\cap X^*\) is the character group of a maximal torus of the semisimple quotient \(\bar G := \factor{G}{\mathcal R G}\)
and the isogeny \(\mathcal DG\to \bar G\) exhibits \(\Q \simpleR\cap X^*\) as a finite index subgroup of \(X^*_{\mathcal DG}\).
\end{rem}

\subsection{Representation rings of parabolic subgroups}\label{sec:rep-of-P}
From now on, we assume that \(G\) is semisimple and simply connected, so that \(\simpleR^\vee\) is a basis of \(X_*\). Let \(\left\{\omega_{\sigma}\mid\sigma\in\simpleR\right\}\) be the fundamental weights of \(G\), defined by
\[
\pairing{\omega_\sigma}{\nu^\vee} = \delta_{\sigma\nu}
\]
for \(\sigma,\nu\in\simpleR\) (\(\delta_{\sigma\nu}\) the Kronecker delta). Our assumptions ensure that the \(\omega_{\sigma}\) form a basis of \(X^*\) \cite{Jantzen}*{II.1.6}.

The parabolic subgroups of \(G\) are classified up to conjugation by subsets \mbox{\(\thetaR\subset\simpleR\)}.
We write \(P_\thetaR\) for the standard parabolic subgroup corresponding to \(\thetaR\). Let \(L_\thetaR\) be its Levi subgroup and \(U_\thetaR\) its unipotent radical, so that
\[
P_\thetaR= L_\thetaR \ltimes U_\thetaR.
\]
The Levi subgroup \(L_{\thetaR}\) is the subgroup of \(G\) generated by the maximal torus \(T\) and the root subgroups corresponding to the roots in \(R_\thetaR:=R\cap\Z\thetaR\), while \(U_\thetaR\) is the subgroup generated by root subgroups corresponding to the roots in \((-R^+)-R_\thetaR\) \cite{Jantzen}*{II.1.7--1.8}.
 For example, \(L_{\emptyset} =  T\) and \(P_{\emptyset}\) is a Borel subgroup, while \(L_\simpleR = P_\simpleR = G\).

\begin{lem}The projection \(P_{\thetaR}\twoheadrightarrow L_{\thetaR}\) induces an isomorphism
  \[
  \Rep(P_{\thetaR})\xleftarrow{\cong} \Rep(L_{\thetaR}).
  \]
\end{lem}
\begin{proof}
  This is a general fact for extensions of unipotent groups.
  For an arbitrary extension \(1\to U\to P\xrightarrow{\pi} L\to 1\),
  the projection \(\pi\) induces a monomorphism \(\Rep(P)\hookleftarrow \Rep(L)\) whose image is generated by those simple \(P\)-modules which have a non-zero \(U\)-fixed vector \cite{Jantzen}*{I.6.3}.
  When \(U\) is unipotent, any non-zero representation has such a fixed vector.
\end{proof}

The Levi subgroup \(L_{\thetaR}\) is a split reductive group with maximal torus \(T\), simple roots \(\thetaR\) and Weyl group the subgroup \(\Weyl[\thetaR]\subset \Weyl\) generated by the reflections in the elements of \(\thetaR\) \cite{Jantzen}*{I.1.7}.
The assumption that \(G\) is simply connected ensures that \(\mathcal DL_{\thetaR}\) is also simply connected.
We can therefore apply \Cref{reps-of-reductive} to \(L_{\thetaR}\) and the decomposition
\begin{equation}
  X^* = (\thetaR^\vee)^\perp \oplus \Omega_\thetaR,
\end{equation}
where  \(\Omega_{\thetaR}\) is the subgroup of \(X^*\) generated by the fundamental weights \(\left\{\omega_\teta\mid\teta\in\thetaR\right\}\).
\begin{cor}\label{repring-of-P}
  Let \(P_{\thetaR}\) be a standard parabolic subgroup of a semisimple simply connected algebraic group \(G\) with fundamental weights \(\left\{\omega_{\sigma} \mid \sigma\in \simpleR\right\}\). Then we have a ring isomorphism
\[
\Z[w_\teta,x_\beta^{\pm 1} \;\vert\; \teta\in\thetaR, \beta\in\nthetaR] \xrightarrow{\quad\cong\quad} \Z[X^*]^{\Weyl[\thetaR]} \cong \Rep(P_{\thetaR})
\]
sending \(w_\teta\) to \(S(e^{\omega_\teta})\) and \(x_\beta\) to \(S(e^{\omega_\beta}) = e^{\omega_\beta}\).
\end{cor}
\begin{proof}
  The fundamental weights \(\left\{\omega_\beta\mid\beta\in\nthetaR\right\}\) form a basis of the character group
  \(X^*(L_{\thetaR}) = (\thetaR^\vee)^\perp\), while the remaining fundamental weights
  \(\left\{\omega_\teta\mid\teta\in\thetaR\right\}\) form a basis of \(\Omega_{\thetaR}\).
  Moreover, the \(\omega_\teta\) form a basis of the free semigroup  \(\FWC[\thetaR] \cap \Omega_{\thetaR}\),
  where \(\FWC[\thetaR]\) is the set of dominant weights of \(L_{\thetaR}\):
  \[
  \FWC[\thetaR] := \left\{ \delta\in X^*\mid \pairing{\delta}{\teta^\vee} \geq 0\text{ for all }\teta\in\thetaR\right\}\qedhere
  \]
\end{proof}

\subsection{Root data and fundamental weights}\label{sec:root-data}
The relation between \(L_\thetaR\), its derived group \(\mathcal DL_\thetaR\) and its semisimple quotient
\(\bar L_\thetaR := \factor{L_\thetaR}{\mathcal RL_\thetaR}\) will be central to all that follows.
We will write \(i\) and \(p\) for the inclusion of \(\mathcal DL_\thetaR\) and the projection onto \(\bar L_\thetaR\):
\[
\mathcal DL_\thetaR \overset{i}\hookrightarrow L_\thetaR \overset{p}\twoheadrightarrow \bar L_\thetaR
\]
Let \(R_\thetaR := R\cap \Z\thetaR\) and \(R_\thetaR^\vee := R^\vee\cap \Z\thetaR^\vee\).
The root data of \(L_\thetaR\), \(\mathcal D L_\thetaR\) and \(\bar L_\thetaR\) can be summarized as follows (\cf \cite{Jantzen}*{II.1.18}):
\begin{alignat*}{12}
  &\mathcal DL_\thetaR\colon\quad &&(\factor{X^*}{(\thetaR^\vee)^\perp},\; &&X_*\cap \Q\thetaR^\vee,          &&i^*R_\thetaR,\; &&R_\thetaR^\vee      &&)\\
  &L_\thetaR\colon                &&(X^*,                                &&X_*,                             &&R_\thetaR,        &&R_\thetaR^\vee      &&)\\
  &\bar L_\thetaR\colon           &&(X^*\cap\Q\thetaR,                   &&\factor{X_*}{\thetaR^\perp},\; &&R_\thetaR,        &&p_*R_\thetaR^\vee &&)
\end{alignat*}
In the proof of \Cref{repring-of-P}, we used the decomposition \(X^*=(\thetaR^\vee)^\perp \oplus \Omega_\thetaR\)
for which \(i^*\) restricts to an isomorphism \(\Omega_\thetaR\cong X^*_{\mathcal DL_\thetaR}\).
The fundamental weights of the simply connected group \(\mathcal DL_\thetaR\) are precisely the images
of the fundamental weights \(\omega_\teta\in\Omega_\thetaR\):
\begin{equation}\label{eq:fundamental-to-fundamental}
  i^*\omega_\teta = \omega_{i^*\teta}
\end{equation}
As in \Cref{choice-of-complement}, we also have a rational decomposition
\begin{equation}\label{eq:Q-decomposition}
  X^*_\Q = (\thetaR^\vee)^\perp_\Q \oplus \Q\thetaR.
\end{equation}
The subspace \(\Q\thetaR\subset X^*_\Q\) does not in general contain any fundamental weights of \(G\).
Rather, the fundamental weights \(\bar \omega_\teta\) of \(\bar L_\thetaR\) are the components of the weights
\(\omega_\teta\) in the direction of \(\Q\thetaR\):
\begin{equation}\label{eq:fundamental-components}
  \omega_\teta = \prii(\omega_\teta) + \bar \omega_\teta
\end{equation}
where \(\prii\colon (\thetaR^\vee)^\perp_\Q \oplus \Q\thetaR\twoheadrightarrow (\thetaR^\vee)^\perp_\Q \) denotes the projection away from \(\Q\thetaR\).

All three groups \(L_\thetaR\), \(\mathcal DL_\thetaR\) and \(\bar L_\thetaR\) have the same Weyl group \(\Weyl[\thetaR]\).
The projection \(i^*\) and the inclusion \(p_*\) are equivariant with respect to its action.

\subsection{Involutions on the weight lattice}\label{sec:involutions-on-weights}
We can analyse the involution on \(\Rep(P_{\thetaR})\) induced by the duality on representations by considering the inclusion
\(\Rep(R_{\thetaR})\hookrightarrow \Rep(T) = \Z[X^*]\).
On \(X^*\), the duality corresponds to multiplication by \( -1 \),
so the induced involution \(*\) on \(\Z[X^*]\) is given by
\begin{align*}
  (e^\xi)^* &:= e^{-\xi}.
\end{align*}
To describe how this involution acts on \(S(\delta)\) for a dominant weight \(\delta\in\FWC[\thetaR]\), we introduce a second involution \(\ddual\) on \(X^*\) that restricts to \(\FWC[\thetaR]\).
\begin{defn}
Let \(w_0\) be the longest element of the Weyl group \( W_{\thetaR}\).
For any \(\xi\in X^*,\) we define
\[
\xi^\ddual:= -w_0.\xi.
\]
\end{defn}
The involution \(\ddual\) restricts to \(\FWC[\thetaR]\) since \(w_0\) sends \(\FWC[\thetaR]\) to \(-\FWC[\thetaR]\) \cite{Jantzen}*{II.1.5}.
As \(S(\xi)^* = S(-\xi) = S(-w.\xi)\) for arbitrary \(\xi\in X^*\) and \(w\in\Weyl[\thetaR]\), we obtain:
\begin{lem}\label{dual-S}
  The \(*\)-dual of \( S(\delta)\in\Z[\WLattice]^{\Weyl[\thetaR]} \) is given by
  \mbox{\(S(\delta)^* = S(\delta^\ddual)\)}.
\end{lem}
The lemma shows that the duality on \(\Z[\WLattice]^{\Weyl[\thetaR]}\) is determined by the action of an element \(w_0\in\Weyl[\thetaR]\).
We therefore take a closer look at the interaction of the Weyl group with the decomposition
\(X^* = (\thetaR^\vee)^\perp\oplus\Omega_\thetaR\).
On \((\thetaR^\vee)^\perp\), the action of \(\Weyl[\thetaR]\) is trivial.
But the action does not restrict to \(\Omega_\thetaR\).
Instead, we can define an action on \(\Omega_\thetaR\) by lifting the action on \(X^*_{\mathcal DL_\thetaR}\)
via the isomorphism  \(i^*\colon \Omega_\thetaR \cong X^*_{\mathcal DL_{\thetaR}}\):
\begin{defn}
  For \(w\in\Weyl[\thetaR]\) and \(\omega\in\Omega_\thetaR\), we define \(w\# \omega\) and \(\omega^\#\)  to be the  unique elements of \(\Omega_\thetaR\) satisfying:
  \begin{align*}
    i^*(w\#\omega)  &= i^*(w.\omega)     &&(= w.i^*\omega)\\
    i^*(\omega^\#)  &= i^*(\omega^\ddual) &&(= i^*(\omega)^\ddual)
  \end{align*}
\end{defn}
Evidently, \(\omega^\# = -w_0\#\omega\), and \(\#\) defines an involution on \(\Omega_\thetaR\) that restricts to \(\FWC_\thetaR\cap\Omega_\thetaR\).

Now consider the rational decomposition
\(
X^*_\Q = (\thetaR^\vee)^\perp_\Q \oplus \Q \thetaR
\).
Unlike the integral decomposition above, this decomposition is equivariant: the action of \(\Weyl[\thetaR]\) restricts to both factors.
Recall that \(\prii\colon X^*_\Q \twoheadrightarrow (\thetaR^\vee)^\perp_\Q\) denotes the projection away from \(\Q\thetaR\).
\begin{lem}
  For any \(w\in\Weyl[\thetaR]\) and any \(\omega\in\Omega_\thetaR\) we have
  \begin{equation}\label{eq:relation-of-actions}
  w\#\omega = w.\omega - \prii(\omega) + \prii(w\#\omega).
  \end{equation}
\end{lem}
\begin{proof}
  As the decomposition \eqref{eq:Q-decomposition} is equivariant, we have \(\prii(\omega) = \prii(w.\omega)\). We can thus rewrite \cref{eq:relation-of-actions} as \(w\#\omega-\prii(w\#\omega) = w.\omega-\prii(w.\omega)\), with both sides in \(\Q\thetaR\). It suffices to check this equation  after applying \(i^*\colon X^* \to X^*_{\mathcal DL_{\thetaR}}\) since \(i^*\) induces an isomorphism from \(\Q\thetaR\) to \(X^*_{\mathcal DL_{\thetaR},\Q}\). As \(i^*\prii = 0\), we are left with the defining equation of \(\w\#\omega\).
\end{proof}

\newcommand{\twist}[1]{\tau(#1)}
\begin{cor}
  The \(\ddual\)-dual of an element \(\omega\in \Omega_\thetaR\) can be written as
  \begin{equation}\label{eq:duality-formula}
    \omega^\ddual = \omega^\# + \twist{\omega},
  \end{equation}
  where \(\twist{\omega} := -\prii(\omega) - \prii(\omega^{\#}) \in (\thetaR^\vee)^\perp\).
  \qed
\end{cor}

Let us reformulate \cref{eq:duality-formula} in terms of the fundamental weights.
Consider first the involution \(\ddual = -w_0\) on the lattice \(X^*_{\mathcal DL_{\thetaR}}\)
of the semisimple group \(\mathcal DL_{\thetaR}\).
It restricts to the set of positive roots \cite{Jantzen}*{II.1.5} and hence to the simple roots \(i^*\thetaR\).
It follows that \(\ddual\) also restricts to an involution on the set of fundamental weights \(\left\{\omega_{i^*\teta} \mid \teta\in\thetaR\right\}\), such that
\begin{equation}\label{eq:simple-dual-fundamental-dual}
  (\omega_{i^*\teta})^\ddual = \omega_{i^*\teta^\ddual}.
\end{equation}
Now consider the action of \(-w_0\) on \(X^*\).
It also restricts to an involution \(\teta\mapsto\teta^\ddual\) on the simple roots \(\thetaR\),
identical under \(i^*\) to the involution on \(i^*\thetaR\).
But it does not restrict to an involution of the fundamental weights \(\omega_\teta\). Rather, we have:
\begin{lem}
  For any \(\teta\in\thetaR\), \(\omega_\teta^\# = \omega_{\teta^\ddual}\).
\end{lem}
\begin{proof}
  Clearly \(\omega_{\teta^\ddual}\in\Omega_{\thetaR}\). So we only need to show that \(\omega_{\teta^\ddual}\) satisfies the defining property of \(\omega_\teta^\#\).  This follows from \cref{eq:fundamental-to-fundamental,eq:simple-dual-fundamental-dual}: \(i^*(\omega_{\teta^\ddual}) = \omega_{i^*\teta^\ddual} =  (\omega_{i^*\teta})^\ddual = (i^*\omega_\teta)^\ddual\), as required.
\end{proof}

\begin{cor}\label{dual-fundamental}
  The \(\ddual\)-duals of fundamental weights \(\omega_\teta\) with \(\teta\in\thetaR\) and \(\omega_\beta\) with \(\beta\in\nthetaR\) are given by
  \begin{align*}
    \omega_\beta^\ddual &= -\omega_\beta\\
    \omega_\teta^\ddual &= \omega_{\teta^\ddual} + \twist{\omega_\teta},
  \end{align*}
  where \(\twist{\omega_\teta}= -\prii(\omega_\teta)-\prii(\omega_{\teta^\ddual})\).
  \qed
\end{cor}

\subsection{The involution on the representation ring}\label{sec:involution-on-RepP}
We can now indicate the general form of the duality on \(\Rep(P_{\thetaR})\).  Let us write the twists \(\twist{\omega_\teta}\) as vectors in the basis \(\left\{\omega_\beta \mid \beta\in\nthetaR\right\}\) of \((\thetaR^\vee)^\perp\):
\[
\twist{\omega_\teta} = \sum_{\beta\in\nthetaR} m_\beta^\teta \omega_\beta \quad \quad \in (\thetaR^\vee)^\perp.
\]

\begin{cor}\label{dual-on-repring}
  Under the isomorphism
  \[
  \Z[w_\teta,x_\beta^{\pm 1} \;\vert\; \teta\in\thetaR, \beta\in\nthetaR] \cong \Rep(P_{\thetaR})
  \]
  of \Cref{repring-of-P}, the involution on \(\Rep(P_{\thetaR})\) that sends a representation to its dual representation corresponds to the duality \(\star\) given by:
  \begin{align*}
    w_\teta^\star &= w_{\teta^\ddual}\cdot \textstyle\prod_{\beta}x_\beta^{m_\beta^\teta} \\
    x_\beta^\star &= x_\beta^{-1}
  \end{align*}
\end{cor}
In the following, we will sometimes simply write \(w_\teta^\star = \omega_{\teta^\ddual}\cdot x^{\vec m^\teta}\), \ie we use the  conventions \(\vec m^\teta := (m^\teta_\beta)_{\beta\in\nthetaR}\) and \(\cramped{x^{\vec m^\teta} := \prod_{\beta}x_\beta^{m_\beta^\teta}}\).
\begin{proof}
  Recall that under the given isomorphism, \(w_\teta\) corresponds to \(S(\omega_\teta)\) and \(x_\beta\) corresponds to \(S(\omega_\beta) = e^{\omega_\beta}\). By \Cref{dual-S} and \Cref{dual-fundamental} we have \(S(\omega_\teta)^* = S(\omega_\teta^\ddual) = S(\omega_{\teta^\ddual} + \twist{\omega_\teta})\).  In general, we have the formula \(S(\delta+\delta') = S(\delta)S(\delta') + (\text{smaller terms})\).
However, when the Weyl group acts trivially on one of \(\delta,\delta'\), there are no smaller terms and we simply have \(S(\delta+\delta') = S(\delta)S(\delta')\). Thus, given that \(\Weyl[\thetaR]\) acts trivially on \((\thetaR^\vee)^\perp\), we obtain
\begin{align*}
S(\omega_{\teta^\ddual} + \twist{\omega_\teta})
&= S(\omega_{\teta^\ddual})\cdot S(\twist{\omega_\teta})\\
&= S(\omega_{\teta^\ddual})\cdot S(\textstyle\sum_{\beta\in\nthetaR} m_\beta^\teta \omega_\beta)\\
&= S(\omega_{\teta^\ddual})\cdot\prod S(\omega_\beta)^{m_\beta^\teta}.\qedhere
\end{align*}
\end{proof}

\encouragepagebreak
\section{Tate cohomology}
This section contains the heart of our computations.
After introducing a notion of twisted Tate cohomology in \Cref{sec:twisted-tate}, in \Cref{sec:tate-of-twisted} we compute this cohomology for any `twisted polynomial ring', a ring with involution that has the same formal description as the representation ring of a parabolic subgroup (\cf \Cref{dual-on-repring}).  In particular, we will see that the twisted Tate cohomology groups are either zero or isomorphic to the untwisted ones.  In \Cref{sec:tate-of-quotients}, this dichotomy is generalized to the Tate cohomology of certain quotients of twisted polynomial rings.  The result applies in particular to the K-ring of a flag variety (\cf \Cref{dichotomy-for-tate}).

\subsection{Twisted Tate cohomology}\label{sec:twisted-tate}
We use the same terminology for rings with involution as in \cite{Me:KOFF}.
In particular, a \deffont{\star-module} is an abelian group equipped with an involution \(\star\) which is a group isomorphism, a \deffont{\star-ring} is a commutative unital ring equipped with an involution \(\star\) which is a ring isomorphism, and a \deffont{\star-ideal} in a \star-ring \( {A} \) is an ideal preserved by \star. We say that a {\star}-ideal \(\ideal a\subset A\) is generated by certain elements \(a_1,\dots,a_n\) if it is generated as an ideal in the usual sense by the elements \(a_1,a_1^\star,\dots,a_n,a_n^\star\).

More generally, given a \star-ring \((A,\star)\), an \deffont{\((A,\star)\)-module} is an \(A\)-module \(M\) together with an additive involution \(\star\) that satisfies \((am)^\star = a^\star m^\star\) for all \(a\in A\) and \(m\in M\).  Thus, any \star-module is a \((\Z,1)\)-module, where \(1\) denotes the trivial involution on \(\Z\), and any \star-ideal in \((A,\star)\) is an \((A,\star)\)-module.

An element \(x\) of a \star-module \((M,\star)\) is \deffont{\star-self-dual} if \(x^*=x\) and \deffont{\star-anti-self-dual} if \(x^*=-x\). The \deffont{Tate cohomology} groups of \((M,\star)\) can be defined ad hoc as follows:
\begin{align*}
   h^+(M,\star)&:=\frac{\ker(\id-\;\star)}{\im(\id+\;\star)}\\
   h^-(M,\star)&:=\frac{\ker(\id+\;\star)}{\im(\id-\;\star)}
\end{align*}
We write \(h^\total(M,\star)\) for the \(\Z/2\)-graded abelian group \(h^+(M,\star)\oplus h^-(M,\star)\).

Given an \((A,\star)\)-module \(M\) and a unit \(l\in A^\times\) such that \(l^* = l^{-1}\), we  define the \deffont{\(l\)-twisted involution} \(\star_l\) on \(M\) by
\[
m^{\star_l} := m^*l.
\]
\begin{lem}\label{tate-rings-and-modules}
Let \((A,\star)\) be a \star-ring, and let \(l_1,l_2\in A^\times\) be units such that \(l_i^* = l_i^{-1}\).
Then for any \((A,\star)\)-module \((M,\star)\), the \(A\)-module structure on \(M\) induces a homomorphism of \(\Z/2\)-graded groups
\begin{align*}
h^\total(A,\star_{l_1}) \otimes h^\total(M,\star_{l_2}) &\to h^\total(M,\star_{l_1l_2})\\
[a]\quad\otimes\quad[m]\quad\quad&\mapsto\quad\quad[am]
\end{align*}
In particular, \(h^\total(A,\star)\) is a \(\Z/2\)-graded ring, and \(h^\total(M,\star_{l_1})\) is a graded \(h^\total(A,\star)\)-module.\qed
\end{lem}

The rings and modules we will be considering in the following will usually come equipped with some `default' involution \(\star\).
We then simply write
\begin{align*}
  h^\total(A) &:= h^\total(A,\star)\\
  h^\total(M) &:= h^\total(M,\star)
  &
  h^\total(M,l) &:= h^\total(M,\star_l)
\end{align*}
and refer to \(h^\total(M,l)\) as the \deffont{\(l\)-twisted Tate cohomology} of \(M\).  The following lemmas will be used frequently.

\begin{lem}\label{tate-exact-sequence}
  Let
  \(
  0 \to (M',\star) \hookrightarrow (M,\star) \xrightarrow{p} (M'',\star) \to 0
  \)
  be a short exact sequence of \star-modules, \ie a sequence of \star-modules and -morphisms such that the underlying sequence of abelian groups is exact.  Then we have an associated six-periodic long exact sequence of Tate cohomology groups
  \[
   \xymatrix{
   {h^+(M')} \ar[r] & {h^+(M)} \ar[r] & {h^+(M'')} \ar[d]^{\partial} \\
   \ar[u]^{\partial} {h^-(M'')} & \ar[l] {h^-(M)} & \ar[l] {h^-(M')}
   }
  \]
  with boundary map given by
  \[
  \partial[p(m)] = \begin{cases}
    [m-m^\star] & \text{ for } [p(m)]\in h^+(M'')\\
    [m+m^\star] & \text{ for } [p(m)]\in h^-(M'').
  \end{cases}
  \]
\end{lem}
We may alternatively view the long exact Tate cohomology sequence as a three-periodic sequence of \(\Z/2\)-graded groups
  \begin{equation*}
    \dots \xrightarrow{\partial} h^\total(M') \to h^\total(M) \to h^\total(M'') \xrightarrow{\partial} h^\total(M') \to h^\total(M) \to h^\total(M'') \xrightarrow{\partial} \dots
  \end{equation*}
  with boundary map \(\partial\) of degree one.
  If \((M',\star)\), \((M,\star)\) and \((M'',\star)\) are modules over some \star-ring \((A,\star)\), and if the morphisms in the short exact sequence are morphisms of \(A\)-modules, then this long exact sequence is a sequence of graded \(h^\total(A,\star)\)-modules.

\begin{lem}\label{tate-tensor-decomposition}
Let \(M\) and \(N\) be \star-modules.  If one of \(M\) and \(N\) is free as an abelian group, then the canonical map
\[
h^\total(M)\otimes h^\total(N) \rightarrow h^\total(M\otimes N)
\]
is an isomorphism of \(\Z/2\)-graded abelian groups.
\end{lem}
\begin{proof}
  The lemma is easily checked in case \(M\) is one of the following three \star-modules:  \(\Z\) with trivial involution, \(\Z\) with involution given by multiplication with \(-1\), \(\Z\oplus\Z\) with involution interchanging the two summands.  In general, any \star-module which is free as an abelian group is a direct sum of such simple \star-modules \cite{Bousfield:1990}*{Prop.~3.7}.
\end{proof}

\begin{rem}\label{rem:tate}
  More generally, for any finite group \(\Gamma\) acting on an abelian group \(M\), Tate introduced cohomology groups \(\hat H^i(\Gamma,M)\) indexed by \(i\in\Z\).  When \(\Gamma\) is cyclic, the cup product with a generator of \(H^2(\Gamma,\Z)\) (\(\Gamma\) acting trivially on \(\Z\)) induces isomorphisms \[\hat H^i(\Gamma,M)\cong \hat H^{i+2}(\Gamma,M),\] so the total Tate cohomology group \(\hat H^\total(\Gamma,M)\) may be collapsed to a \(\Z/2\)-graded group \cite{CartanEilenberg}*{Chapter~XII, \S\S~7 and 11}.  For \(\Gamma=\Z/2\), this yields the \(\Z/2\)-graded group defined above.  \Cref{tate-rings-and-modules,tate-exact-sequence} may as easily be checked directly as they may be obtained by specialization from more general results in the literature.
\end{rem}

\subsection{Tate cohomology of twisted polynomial rings}\label{sec:tate-of-twisted}
We use the term \deffont{twisted polynomial ring} to refer to a \star-ring \((A,\star)\) of the following form:
\begin{compactitem}
\item As a ring, \(A\) is isomorphic to a tensor product of a polynomial ring and a ring of Laurent polynomials
  \[ A \cong \Z\left[\gamma_i,\gamma_i', \mu_j, x_k^{\pm 1} \mid i\in I, j\in J, k\in K\right] \]
\item This isomorphism can be chosen such that the duality on \(A\) has the following form:
  \begin{align*}
    \gamma_i^\star &= \gamma_i'\cdot x^{\vec c^i} && \text{ for certain \(\vec c^i \in \Z^{\abs K} \)} &&\\
    \mu_j^\star &= \mu_j\cdot x^{\vec m^j}     && \text{ for certain \(\vec m^j \in \Z^{\abs K} \)} &&\\
    x_k^\star &= x_k^{-1}
  \end{align*}
\end{compactitem}

Here and in the following, we use a bold-face letter \(\vec c\) to denote a tuple with entries \(c_k\), and we define
\[
 x^{\vec c} := \prod_k x^{c_k}.
\]
We think of the \(\mu_i\) as `self-dual' generators and of the \(\gamma_i\) as  `non-self-dual' generators, except that they carry `individual  twists' \(\vec c^i\) and \(\vec m^i\).
By \Cref{dual-on-repring}, the representation ring of any parabolic subgroup of a simply connected split semisimple algebraic group is a twisted polynomial ring in our sense:  we only need to take
\begin{align*}
  &\mu_\teta := w_\teta && \text{ for each \(\teta\) with } \teta^\ddual = \teta \quad &&(\text{so } \vec m^\teta = \vec m^\teta)\\
  &\left.\begin{aligned}
      \gamma_\teta &:= w_\teta \\
      \gamma_\teta' &:= w_{\teta^\ddual}
    \end{aligned}
  \right\} &&\text{ for each pair \((\teta,\teta^\ddual)\) with \(\teta^\ddual \neq \teta\) } \quad &&(\text{so } \vec c^\teta = \vec m^\teta)
\end{align*}

Our aim is to describe \( h^\total(A,x^{\vec t}) \) for any twisted polynomial ring \((A,\star)\) as above and any twist \(x^\vec t\) (\(\vec t\in\Z^{\abs K}\)). As we will see, these groups essentially depend only the twists \(\vec m^j \) of the `self-dual' generators and on the global twist \( \vec t \), or rather on their reductions modulo two \(\bar{\vec m}^j \), \(\bar{\vec t} \in (\Z/2)^{\abs K}\).  First, we have a simple criterion when \( h^\total(A,x^{\vec t}) \) vanishes:

\begin{prop}\label{van:prop:vanishing}
  The twisted Tate cohomology module \( h^\total(A, x^{\vec t})\) is non-zero if and only if \(\bar{\vec t}\) is contained in the subspace of \((\Z/2)^{\abs K}\) generated by the \(\bar{\vec m}^j\).
\end{prop}

We can rephrase the proposition as follows. Let \( \mat M:= (\vec m^1, \dots, \vec m^{\abs{J}}) \) be the \(\abs{K}\times\abs{J}\)-matrix whose columns are given by the vectors \(\vec m^j\), and let \( \bar{\mat M} \) be its reduction modulo two.  The proposition says that  \( h^\total(A, x^{\vec t})\) is non-zero if and only if the inhomogeneous linear system of equations
\begin{equation}\label{van:eq:LSEred}
  \bar{\mat M}\cdot{\bar{\vec j}} = \bar{\vec t} \tag{$\overline{\mathrm{I}}$}
\end{equation}
has some solution \(\bar{\vec j}\in (\Z/2)^{\abs J}\).
Note that any vector \(\bar{\vec j}\in (\Z/2)^{\abs J}\) can be lifted to a tuple with coefficients in \(\N_0\), and that there is a unique lift with coefficients in \(\{0,1\}\).  We call the lift \(\vec j\in\{0,1\}^{\abs J}\) \deffont{minimal}.

\begin{prop}\label{van:prop:description}
  Let  \(\bar{\mathcal S}\subset (\Z/2)^{\abs J}\) be a basis of the space of solutions of the homogeneous part of \eqref{van:eq:LSEred}, and let \(\mathcal S\subset \{0,1\}^{\abs J}\) be the set of minimal lifts of these basis elements.
  The untwisted Tate cohomology ring of \( A \) is concentrated in even degree and can be written as
  \[
  h^\total(A )
  = \Z/2\left[\gamma_i\gamma_i^\star \mid i\in I \right]
  \otimes
  \frac{\Z/2\left[\nu_{\vec{j}}, \sigma_j \mid \vec j\in\mathcal S, j\in J\right]}
  {\left(\nu_{\vec j}^2 = \sigma^{\vec j} \mid \vec j \in \mathcal S\right)}
  \]
  with \(\nu_{\vec j} := \mu^{\vec j} x^{\frac{1}{2}\mat M\vec j} \) and \(\sigma_j := \mu_j^2 x^{\vec m^j}\).
\end{prop}
\begin{prop}\label{van:prop:twist}
  If \(h^\total(A,x^{\vec t})\) is non-zero,
  there is an element \(\zeta_{\vec t}\in h^+(A,x^{\vec t})\) such that multiplication with \(\zeta_{\vec t}\) induces an isomorphism
  \[
  h^\total(A)\xrightarrow{\cong}h^\total(A,x^{\vec t}).
  \]
  Moreover, if both \(h^\total(A,x^{\vec t_1})\) and \(h^\total(A,x^{\vec t_2})\) are non-zero, then so is \(h^\total(A,x^{\vec t_1+\vec t_2})\) and we may take \(\zeta_{\vec t_1+\vec t_2} = \zeta_{\vec t_1}\cdot\zeta_{\vec t_2}\).
\end{prop}

\begin{proof}[Proof of \Cref{van:prop:vanishing,van:prop:description}]
  Let \(\star_{\vec t}\) denote the \(x^{\vec t}\)-twisted involution on \(A\).
  As a first reduction, we may factorize \( (A,\star_{\vec t}) \) into a tensor product of \((A,\star)\)-modules:
  \begin{align*}
    (A,\star_{\vec t})
    &= \left( \Z\left[\gamma_i,\gamma_i^{\star_{\vec t}}\mid i\in I \right], \star_{\vec t} \right) \otimes
       \left( \Z\left[\mu_j,x_k^{\pm 1}\mid j\in J, k\in K \right], \star_{\vec t} \right).
    \intertext{%
      As both factors are free as abelian groups, we may apply \Cref{tate-tensor-decomposition} to obtain a corresponding factorization of the Tate cohomology modules:
    }
    h^\total(A,x^{\vec t}) &= \Z/2\left[\gamma_i\gamma_i^*\mid i\in I \right] \otimes h^\total(A',x^{\vec t})
  \end{align*}
  with \(A' := \Z\left[\mu_j, x_k^{\pm 1} \mid j\in J, k\in K\right]\).

  In order to compute the Tate cohomology of \(A'\), we observe that the set of  monomials
  \[
  \left\{ \mu^{\vec j}x^{\vec k} \mid \vec j\in \mathbb{N}_0^{\abs J}, \vec k\in\Z^{\abs K} \right\}
  \]
  forms a basis of \( A' \) over \(\Z\) which is taken into itself by the involution \(\star_{\vec t}\). This implies that \(h^+(A')\) is a \(\Z/2\)-module on those basis elements fixed under the involution, while \(h^-(A')=0\).

  So let \(h:=\mu^{\vec j}x^{\vec k}\) be an arbitrary basis element. We can easily express the fixed-point condition on \( h \) as a linear system of equations in \( (\vec j,\vec k)\in\mathbb{N}_0^{\abs J}\times\Z^{\abs K}\):
  \begin{align*}
    h^{*_{\vec t}} &= h  \notag\\
    \Iff \quad \mat M\vec j - 2\vec k &= -\vec t \tag{$\mathrm{I}$} \label{van:eq:LSE}
  \end{align*}
  Proposition~\ref{van:prop:vanishing} follows from the fact that the linear system \eqref{van:eq:LSE} has a solution if and only if its reduction modulo two \eqref{van:eq:LSEred} has a solution: given a solution \(\bar{\vec j}_0\in(\Z/2)^{\abs J}\) of \eqref{van:eq:LSEred}, we obtain a solution  \( (\vec j_0, \vec k_0) \) of \eqref{van:eq:LSE} by taking \(\vec j_0\) to be some lift of \(\bar{\vec j}_0\) and setting \(\vec k_0 := \frac{1}{2}(\mat M\vec j_0 +\vec t)\).

  Now consider the homogeneous part of the system and its reduction modulo two:
  \begin{align}
    \mat M\vec j - 2\vec k    &= 0 \tag{$\mathrm{H}$} \label{van:eq:LSEhom} \\
    \bar{\mat M} \bar{\vec j} &= 0 \tag{$\overline{\mathrm{H}}$} \label{van:eq:LSEhomred}
  \end{align}
  Again, the space of solutions of \eqref{van:eq:LSEhom} is determined by the space of solutions of \eqref{van:eq:LSEhomred}:
  For any solution \( (\vec j_h, \vec k_h) \) of \eqref{van:eq:LSEhom}, we have \(\vec k_h = \frac{1}{2}\mat M\vec j_h\) and we can write \(\vec j_h\) as \(\vec j_h = \vec j_m + 2\vec r\) for some \(\vec j_m\) with coefficients in \(\{0,1\}\) and some \(\vec r\) with coefficients in \(\N_0\). Then \(\bar{\vec j}_m\) is a solution of \eqref{van:eq:LSEhomred}.
  Conversely, any solution \(\bar{\vec j}_m\) of \eqref{van:eq:LSEhomred} and any \(\vec r\in\N_0^{\abs J}\) determine a solution
  \begin{equation}\label{van:eq:solutions}
  \matrix{\vec j_h \\ \vec k_h} =  \matrix{\vec j_m\\ \frac{1}{2}\mat M\vec j_m} + \matrix{2\vec r\\ \mat M\vec r}
  \end{equation}
  of \eqref{van:eq:LSEhomred}.

  We may view the space of solutions of \eqref{van:eq:LSEhom} as a sub-semigroup of \( \N_0^{\abs J} \times \Z^{\abs K} \). In order to describe generators, we choose a basis \(\bar{\mathcal S}\) of the space of solutions of \eqref{van:eq:LSEhomred} and let \(\mathcal S\) be the set of minimal lifts of these basis vectors, as above. Let \(\vec e_j \) with \( j\in J\) denote the standard basis vectors in \( \N_0^{\abs J} \). Then it follows from \eqref{van:eq:solutions} that a set of generators of the semigroup of solutions of \eqref{van:eq:LSEhom} is given by
  \[
  \left\{ \matrix{\vec j\\ \frac{1}{2}\mat M\vec j}, \matrix{2\vec e_j \\ \vec m^j} \mid \vec j\in\mathcal S, j\in J \right\}
  \]
  These generators correspond to the algebra generators \( \nu_{\vec j}\) and \( \sigma_j \) of \( h^\total(A) \) appearing in \Cref{van:prop:description}.

  It remains to determine the relations between these generators. Suppose
  \[
  \sum_{\vec j\in\mathcal S} a_{\vec j} \matrix{\vec j\\ \frac{1}{2}\mat M\vec j} + \sum_{j\in J}b_j\matrix{2\vec e_j \\ \vec m^j} = 0
  \]
  for certain \( a_{\vec j}, b_j \in \Z \). Since the vectors \(\bar{\vec j}\in\bar{\mathcal S}\) are linearly independent over \(\Z/2\), all coefficients \( a_{\vec j} \) must be even. Conversely, for any choice of even coefficients \( a_{\vec j} \), we obtain a relation by setting
  \[
  b_j := -\left[\sum_{\vec j\in\mathcal S}\frac{a_{\vec j}}{2}\vec j\right]_j
  \]
  where \( [\vec v]_j \) denotes the \(j^{\text{th}}\) component of a tuple \( \vec v \). It follows that the space of relations is generated by the unique relations with \(a_{\vec j}\) equal to \( 2 \) for just one \(\vec j\in\mathcal S\) and equal to \( 0\) for all others, \ie by the relations
  \[
  \left\{ 2\matrix{\vec j\\ \frac{1}{2}\mat M\vec j} = \sum_{j\in J}[\vec j]_j\matrix{2\vec e_j \\ \vec m^j} \mid \vec j\in\mathcal S\right\}.
  \]
  In \( h^\total(A) \), these relations correspond to the relations \( \nu_{\vec j}^2 = \sigma^{\vec j} \).
\end{proof}
\begin{proof}[Proof of \Cref{van:prop:twist}]
  A general solution \( (\vec j_i, \vec k_i) \) of \eqref{van:eq:LSE} will be of the form
  \[
  \matrix{\vec j_i \\ \vec k_i} = \matrix{\vec j_0 \\\vec k_0} + \matrix{\vec j_h \\ \vec k_h},
  \]
  where \( (\vec j_0 ,\vec k_0) \) is some solution of \eqref{van:eq:LSE} with coefficients in \(\{0,1\}\)  and \( (\vec j_h,\vec k_h) \) is a solution of the homogeneous system \eqref{van:eq:LSEhom}.
  Thus, we can take \[
  \zeta_{\vec t} := \mu^{\vec j_0}x^{\vec k_0} = \mu^{\vec j_0}x^{\frac{1}{2}(\mat M\vec j_0 + \vec t)}.
  \]
  The second claim of \Cref{van:prop:twist} also follows from this explicit description of \(\zeta_{\vec t}\).
\end{proof}

\subsection{Tate cohomology of quotients}\label{sec:tate-of-quotients}
In this section, we show how to generalize the results of the previous section to certain quotients of twisted polynomial rings.  More precisely, we show:
\begin{prop}\label{tate-of-quotient}
  Let \((A,\star)\) be a \star-ring,
  and let \(l\in A^\times\) be a unit such that \(l^*= l^{-1}\).
  If \(\zeta\in A\) is \(\star_l\)-self-dual and induces an isomorphism
  \[
  h^\total(A) \xrightarrow[\cdot\zeta]{\cong} h^\total(A,\star_l),
  \]
  then multiplication by \(\zeta\) also induces an isomorphism
  \[
  h^\total\left(\factor{A}{\ideal a}\right) \xrightarrow[\cdot\zeta]{\cong} h^\total\left(\factor{A}{\ideal a},\star_l\right),
  \]
  for any \star-ideal \(\ideal a\subset A\) that can be generated by elements \(\lambda_1,\dots,\lambda_k\) and \star-self-dual elements \(\mu_{1},\dots, \mu_{l}\) such that \(\lambda_1,\lambda_1^*\dots,\lambda_k,\lambda_k^*, \mu_{1},\dots, \mu_{l}\) is a regular sequence in \(A\).
\end{prop}
The proof is based on two lemmas that deal with the special case when \(\ideal a\) is generated, as a \star-ideal, by a single element. If the generator is a self-dual element \(\mu\) of \(A\), we write the ideal as \((\mu)\); if the generator is a non-self-dual element \(\lambda\), we write the ideal as \((\lambda,\lambda^*)\).
\begin{lem}\label{h-principal-ideal}
  Let \( (A,*) \) be a \star-ring, and let \(l\in A^\times\) be a unit such that \(l^*= l^{-1}\).
  \begin{itemize}
  \item If \( \mu \in A \) is a \star-self-dual element that is not a zero divisor, then multiplication by \( \mu \) induces a graded isomorphism
    \[
    h^\total(A,*_l) \xrightarrow[\cdot \mu]{\cong} h^\total((\mu),*_l).
    \]
  \item If \(\lambda, \lambda^*\) is a regular sequence in \( A \), then multiplication by the product \( \lambda\lambda^* \) induces a graded isomorphism
    \[
    h^\total(A,*_l) \xrightarrow[\cdot \lambda\lambda^*]{\cong} h^\total((\lambda,\lambda^*),*_l).
    \]
  \end{itemize}
\end{lem}
\begin{proof}This is a mild generalization of one half of Proposition~4.1 of \cite{Me:KOFF}.  It can be proved in exactly the same way. (The other half of the proposition can also be generalized, but it is not used in the following.)
\end{proof}

\begin{lem}\label{h-induction-start}
  Let \( (A,\star) \) be a \star-ring, and let \(l\in A^\times\) be a unit such that \(l^*= l^{-1}\).
  Suppose there is an \(\star_l\)-self-dual element \(\zeta\in A\) such that multiplication with \(\zeta\) induces an isomorphism
  \[h^\total(A)\xrightarrow[\cdot\zeta]{\cong} h^\total(A,\star_l).\]
  Then multiplication with \(\zeta\) also induces isomorphisms
  \begin{align*}
    h^\total\left(\factor{A}{(\mu)}\right)               &\xrightarrow{\cong} h^\total\left(\factor{A}{(\mu)},\star_l\right) \\
    h^\total\left(\factor{A}{(\lambda,\lambda^*)}\right) &\xrightarrow[\cdot\zeta]{\cong} h^\total\left(\factor{A}{(\lambda,\lambda^*)}, \star_l\right)
  \end{align*}
  for any \star-self-dual non-zero divisor \(\mu\) of \(A\) and any regular sequence of the form \(\lambda,\lambda^*\) in \(A\).
\end{lem}
\begin{proof}
  As \(\zeta\) is \(\star_l\)-self-dual, multiplication by \(\zeta\) induces a \(\star\)-morphism \((M,\star)\to(M,\star_l)\) for any \((A,\star)\)-module \((M,\star)\).  Given a regular sequence \(\lambda,\lambda^*\) as in the \namecref{h-induction-start},
  consider the following commutative square of \((A,\star)\)-modules and the induced commutative square of \(h^\total(A)\)-modules:
  \[\begin{aligned}
    \xymatrix{
      {(A,\star)} \ar[r]^{\cdot\lambda\lambda^*}\ar[d]_{\cdot\zeta} & {((\lambda,\lambda^*),\star)} \ar[d]^{\cdot\zeta} \\
      {(A,\star_l)} \ar[r]_{\cdot\lambda\lambda^*} & {((\lambda,\lambda^*),\star_l)}
    }
  \end{aligned}
  \quad\leadsto\quad
  \begin{aligned}
    \xymatrix{
      {h^\total(A)} \ar[r]^{\cdot\lambda\lambda^*}_{\cong}\ar[d]_{\cdot\zeta}^{\cong} & {h^\total((\lambda,\lambda^*))} \ar[d]^{\cdot\zeta} \\
      {h^\total(A,\star_l)} \ar[r]_{\cdot\lambda\lambda^*}^{\cong} & {h^\total((\lambda,\lambda^*),\star_l)}
    }
  \end{aligned}
  \]
  In the right-hand square, the horizontal morphisms are isomorphisms by \Cref{h-principal-ideal} and the vertical morphism on the left is an isomorphism by assumption.  So \(\zeta\colon h^\total((\lambda,\lambda^*))\to h^\total((\lambda,\lambda^*),\star_l)\) is also an isomorphism.
  To obtain the analogous claim for the quotient \(A/(\lambda,\lambda^*)\), we consider the long exact Tate cohomology sequences associated with the obvious two short exact sequences of \((A,\star)\)-modules:
  \begin{align*}
    0\to ((\lambda,\lambda^*),\star)    \to &(A,\star)   \to (\factor{A}{(\lambda,\lambda^*)}, \star)   \to 0 \\
    0\to ((\lambda,\lambda^*),\star_l) \to &(A,\star_l) \to (\factor{A}{(\lambda,\lambda^*)}, \star_l) \to 0
  \end{align*}
  Multiplication by \(\zeta\) induces an isomorphisms on two out of three terms of these long exact sequences, so by the Five Lemma it also induces an isomorphism on the remaining term.
  The proof for the case of a \star-self-dual non-zero divisor \(\mu\) works in the same way.
\end{proof}

\begin{proof}[Proof of \Cref{tate-of-quotient}]
\Cref{tate-of-quotient} now follows by induction over the number of generators of \(\ideal a\).  Indeed, suppose the claim holds for the ideal \(\ideal a'\subset \ideal a\) generated by all but one of the given generators of \(\ideal a\), so that \(\zeta\) induces an isomorphism \(h^\total(A/\ideal a')\cong h^\total(A/\ideal a', \star_l)\). The assumptions ensure that the remaining generator of \(\ideal a\) will either be a \star-self-dual element \(\mu\) which is not a zero-divisor in \(A/\ideal a'\) or an element \(\lambda\) such that \(\lambda,\lambda^*\) is a regular sequence in \(A/\ideal a'\).  Thus, we can conclude using \Cref{h-induction-start}.
\end{proof}

\encouragepagebreak
\section{Translation}
\subsection{The proofs}
The considerations above easily imply the following two precursors of \Cref{mainthm:iso}:

\begin{prop}\label{dichotomy-for-P}
  Let \(P\) be a parabolic subgroup of a semisimple simply connected algebraic group over a field,
  and let \(\Rep(P)\) be its representation ring equipped with the usual involution induced by duality.
  Let \(\omega\in\Rep(P)^\times\) be a character of \(P\).
  If the twisted total Tate cohomology group \(h^\total(\Rep(P),\omega)\) is non-zero, then there exists an element \(\zeta_{\omega}\) in \(h^+(\Rep(P),\omega)\) such that multiplication by \(\zeta_{\omega}\) induces an isomorphism of \(\Z/2\)-graded modules
  \[
  h^\total\left(\Rep(P)\right) \xrightarrow{\cong} h^\total\left(\Rep(P),\omega\right).
  \]
  Moreover, if both \(h^\total(\Rep(P),\omega_1)\) and \(h^\total(\Rep(P),\omega_2)\) are non-zero for characters \(\omega_1\) and \(\omega_2\) of \(P\), then so is \(h^\total(\Rep(P),\omega_1\otimes \omega_2)\) and we may choose \(\zeta_{\omega_1\otimes \omega_2} = \zeta_{\omega_1}\cdot \zeta_{\omega_2}\).
\end{prop}
\begin{proof}
  This is a translation of \Cref{van:prop:description,van:prop:twist}.
  Indeed, by \Cref{repring-of-P,dual-on-repring} we haven an isomorphism
  \begin{equation}
  \Rep(P)\cong\Z[w_\teta,x_\beta^{\pm 1} \;\vert\; \teta\in\thetaR, \beta\in\nthetaR]
  \end{equation}
  that identifies \(\Rep(P)\) with a twisted polynomial ring in the sense of \Cref{sec:tate-of-twisted}.
  Under this isomorphism, the character \(\omega_\beta\) of \(P\) (\(\beta\in\nthetaR\)) corresponds to the invertible element \(x_\beta\),
  so a general character \(\omega^{\vec t} =\otimes_{\beta\in\nthetaR}\omega^{t_\beta}_\beta\) corresponds to a product \(x^{\vec t} = \prod_{\beta\in\nthetaR} x_\beta^{t_\beta}\).
\end{proof}

\begin{thm}\label{dichotomy-for-tate}
  Let \(G\) be a split semisimple simply connected algebraic group over a field,
  let \(P\subset G\) be a parabolic subgroup,
  and let \(\lb L\) be a line bundle over \(G/P\).
  If the twisted total Tate cohomology group \(h^\total(G/P,\lb L)\) is non-zero, then there exists an element \(\zeta_{\lb L}\) in \(h^+(G/P,\lb L)\) such that multiplication by \(\zeta_{\lb L}\) induces an isomorphism of \(\Z/2\)-graded modules
  \[
  h^\total\left(\factor{G}{P}\right) \xrightarrow{\cong} h^\total\left(\factor{G}{P},\lb L\right).
  \]
  Moreover, if both \(h^\total(G/P,\lb L_1)\) and \(h^\total(G/P,\lb L_2)\) are non-zero for line bundles \(\lb L_1\) and \(\lb L_2\), then so is \(h^\total(G/P,\lb L_1\otimes \lb L_2)\) and we may choose \(\zeta_{\lb L_1\otimes \lb L_2} = \zeta_{\lb L_1}\cdot \zeta_{\lb L_2}\).
\end{thm}
\begin{proof}
  We apply \Cref{tate-of-quotient} to \(A := \Rep(P)\) and the ideal \(\ideal a\) of \Cref{thm:Hodgkin}.
  The required assumptions on \(A\) are satisfied by the previous \namecref{dichotomy-for-P}.
  To check that \(\ideal a\) also satisfies the required assumptions, we recall that the representation ring of \(G\)
  can be written as a polynomial ring on generators \(\lambda_1,\lambda_1^*,\cdots, \lambda_k,\lambda_k^*, \mu_1,\dots, \mu_l\)
  with \(\mu_1,\dots,\mu_l\) self-dual (\eg take \(\thetaR = \simpleR\) in \Cref{dual-on-repring}).
  The ideal \(\ideal a\) is generated by the corresponding rank zero classes \(\lambda_i-\rank(\lambda_i),\dots\) in \(\Rep(P)\).
  These classes form a regular sequence in \(\Rep(G)\).
  By \Cref{thm:Steinberg}, \(\Rep(P)\) is free over \(\Rep(G)\), so their images in \(\Rep(P)\) likewise form a regular sequence.
\end{proof}

\Cref{mainthm:iso} is now an immediate corollary of \Cref{thm:W=h}.
Similarly, \Cref{mainthm:vanishing} is a consequence of:
\begin{thm}\label{vanishing-for-tate}
  Let \(G/P\) and \(\lb L\in\Pic(G/P)\) be as above.
  Let \(\vec m^\teta\) be the twist vectors describing the duality on \(\Rep(P)\) as in \Cref{dual-on-repring}.  View the twist vectors as elements of \(\Pic(G/P)\) via the identification of the Picard group with the character group \(X^*(P)\) as in \eqref{eq:Pic=X}.
  Then \(h^\total(G/P,\lb L)\) is non-zero if and only if the reduction of \(\lb L\) modulo two is a linear combination of the reductions \(\bar{\vec m}^\teta\) of those twist vectors with \(\teta=\teta^\ddual\).
\end{thm}
\begin{proof}
  Again, this is a direct translation of \Cref{van:prop:vanishing}.
\end{proof}

\subsection{The twists}
In this section we describe the twists \(\vec m^\teta\) of \Cref{dual-on-repring} explicitly.
The result in \Cref{dyn:cor:omega} below will refer to a list of values that we have organized in \Cref{fig:twists} in terms of Dynkin diagrams.
Let us recall our graphical conventions and explain our notation.
\begin{itemize}
\item  Associated with our semisimple group \(G\), we have an  `ambient root system' \(\RSystem\)
  with root space \(X^*_\Q\) and simple roots \(\simpleR\).
  The corresponding `ambient Dynkin diagram' has a node for each simple root \(\sigma\in\simpleR\).
  We draw this Dynkin diagram using blank circles, \eg:
  \begin{center}
    \ifdraft{}{
      \begin{tikzpicture}[dynPicture]
        {[dynDiagram]
          \dynR{}
          \dynR{}
          \dynR{}
          \dynR{}
          \dynR{}
          \dynR{}
          \dynR{}
          \dynR[by right arrow]{}
         }
      \end{tikzpicture}
    }
  \end{center}
  We write \(C_{\sigma\nu} = \pairing{\sigma}{\nu^\vee}\) for the coefficients of the Cartan matrix of \(\RSystem\).
  These can be read off directly from the Dynkin diagram using \Cref{fig:Cartan}
  (ignoring the colours for the moment).

\item  The root system \(\RSystem[\thetaR]\) associated with \(L_\thetaR\) or \(\bar L_\thetaR\) has root space \(\Q\thetaR\subset X^*_\Q\) and simple roots \(\thetaR\).
  The corresponding Dynkin diagram has a node for each root \(\teta\in\thetaR\).
  We will indicate this Dynkin diagram by filling the corresponding nodes of the ambient diagram, \eg:
  \begin{center}
    \ifdraft{}{
      \begin{tikzpicture}[dynPicture]
        {[dynDiagram]
             \dynR{}
          {[dynTheta]
            \dynR{}
          }
          \dynR{}
          \dynR{}
          {[dynTheta]
            \dynR{}
            \dynR{}
            \dynR{}
          }
          \dynR[by right arrow]{}
        }
      \end{tikzpicture}
    }
  \end{center}
  The coefficients of the Cartan matrix of \(\RSystem[\thetaR]\) and its inverse matrix will be denoted by \(\bar C_{\teta\nu}\)
  and \(\bar C^{\teta\nu}\), respectively.\footnote{%
    For \(\teta,\nu\in\thetaR\), we have of course \(\bar C_{\teta\nu} = C_{\teta\nu}\), but \(\bar C^{\teta\nu}\) may be different from \(C^{\teta\nu}\).}
  For irreducible \(\RSystem[\thetaR]\), the values \(\bar C^{\teta\nu}\) are conveniently summarized in  \cite{Bourbaki:Lie456}*{Plates~I--IX, entries~(VI)}.
  \Cref{fig:aij} displays the values that will be relevant for us.
  In general, the Cartan matrix and its inverse are block matrices
  with the Cartan matrices of each irreducible component on the diagonal and all off-diagonal blocks zero,
  so the coefficients can be computed one component at a time.
\end{itemize}

\begin{figure}
\begin{multicols}{2}
  \newcommand{\lablebox}[1]{%
  \node[anchor=north west,align=left] at (1,-0.5)
  {\parbox{3cm}{\raggedright\begin{flalign*}#1\end{flalign*}}};
}%
\begin{tikzpicture}[dynPicture]
  \dynTypeL{$A_l$}
  {[dynDiagram,dynTheta]
    \dynR{[lbl=$\teta_1$][right=of type]}
    \dynR{[lbl=$\teta_2$] (a)}
    \dynR{[lbl=$\teta_3$]}
    \dynR[by fadeout line]{[empty]} 
    \dynR[by fadein line]{[lbl=$\teta_{l-1}$] (z)}
    \dynR{[lbl=$\teta_l$]}
  }
  \draw[dynduality] (a) to [bend left=130] (z);
  \lablebox{
    \bar C^{i,j}&=\begin{cases}
      \tfrac{i(l+1-j)}{l+1} & \text{ for } i\leq j \\
      \tfrac{j(l+1-i)}{l+1} & \text{ for } i\geq j 
    \end{cases}
    &
  }
\end{tikzpicture}

\begin{tikzpicture}[dynPicture]
  \dynTypeL{$B_l$}
  {[dynDiagram,dynTheta]
    \dynR{[lbl=$\teta_1$][right=of type]}
    \dynR{[lbl=$\teta_2$]}
    \dynR{[lbl=$\teta_3$]}
    \dynR[by fadeout line]{[empty]} 
    \dynR[by fadein line]{[lbl=$\teta_{l-1}$]}
    \dynR[by right arrow]{[lbl=$\teta_l$]} 
  }
  \lablebox{
    \bar C^{i,1}&=  \begin{cases}
      1            & \text{ for } i<l\\
      \tfrac{1}{2} & \text{ for } i=l
    \end{cases}&
    \\
    \bar C^{i,l}&= \begin{cases}
      i            & \text{ for } i<l\\           
      \tfrac{l}{2} & \text{ for } i=l
    \end{cases}&
  }
\end{tikzpicture} 

\begin{tikzpicture}[dynPicture]
  \dynTypeL{$C_l$}
  {[dynDiagram,dynTheta]
    \dynR{[lbl=$\teta_1$][right=of type]}
    \dynR{[lbl=$\teta_2$]}
    \dynR{[lbl=$\teta_3$]}
    \dynR[by fadeout line]{[empty]} 
    \dynR[by fadein line]{[lbl=$\teta_{l-1}$]}
    \dynR[by left arrow]{[lbl=$\teta_l$]} 
  }
  \lablebox{
    \bar C^{i,1}&=1 &
    \bar C^{i,l}&=\tfrac{i}{2}&
  }
\end{tikzpicture}
\captionsetup{margin=0pt}
\caption{Selected coefficients \mbox{\ensuremath{\bar C^{i,j}:=\bar C^{\teta_i\teta_j}}} of the inverse of the Cartan matrix, as relevant for our computations. The arrows indicate the involution \(\teta\mapsto\teta^\ddual\) on \(\thetaR\).}

\begin{tikzpicture}[dynPicture]
  \dynTypeL{$D_l$}
  {[dynDiagram, dynTheta]
    \dynR{[lbl=$\teta_1$][right=of type]}
    \dynR{[lbl=$\teta_2$]}
    \dynR{[lbl=$\teta_3$]}
    \dynR[by fadeout line]{[empty]} 
    \dynR[by fadein line]{[lbl=$\teta_{l-2}\;\;$](v)}
    \dynR{[label={right:$\teta_l$}][above right=of v] (vv)}
    \dynR[with v]{[label={right:$\teta_{l-1}$}][below right=of v] (w)}
  }
  \draw[dynduality,dashed] (vv) to [bend left=30]  node[auto] {${\;}_{(\text{if } l \text{ odd})}$} (w);
  \lablebox{
    \bar C^{i,1}&= \begin{cases}
      1             & \text{ for } i<l-1  \\
      \tfrac{1}{2}  & \text{ for } i=l,l-1           
    \end{cases}&
    \\
    \bar C^{i,l}&= \begin{cases}
      \tfrac{i}{2}   & \text{ for } i<l-1 \\           
      \tfrac{l-2}{4} & \text{ for } i=l-1\\
      \tfrac{l}{4}   & \text{ for } i=l 
    \end{cases}&
  }
\end{tikzpicture}

\begin{tikzpicture}[dynPicture]
  \dynTypeL{$E_6$}
  {[dynDiagram,dynTheta]
    \dynR{[lbl=$\teta_1$][right=of type]}
    \dynR{[lbl=$\teta_3$] (a) }
    \dynR{[lbl=$\teta_4$](e6)}
    \dynR[with e6 by line]{[label={right:$\teta_2$}][above=of e6]}
    \dynR[with e6 by line]{[lbl=$\teta_5$][right=of e6] (z) }
    \dynR{[lbl=$\teta_6$]}
  }
  \draw[dynduality] (a) to [bend left=120] (z);
  \lablebox{
    \bar C^{1,6}&=\tfrac{2}{3} \quad& \bar C^{3,6}&=\tfrac{4}{3} \quad& \bar C^{5,6}&=\tfrac{5}{3}\\
    \bar C^{2,6}&=1                 & \bar C^{4,6}&=2                 & \bar C^{6,6}&=\tfrac{4}{3}
  }
\end{tikzpicture}

\begin{tikzpicture}[dynPicture]
  \dynTypeL{$E_7$}
  {[dynDiagram,dynTheta]
    \dynR{[lbl=$\teta_1$][right=of type]}
    \dynR{[lbl=$\teta_3$]}
    \dynR{[lbl=$\teta_4$](e6)}
    \dynR[with e6 by line]{[label={right:$\teta_2$}][above=of e6]}
    \dynR[with e6 by line]{[lbl=$\teta_5$][right=of e6]}
    \dynR{[lbl=$\teta_6$]}
    \dynR{[lbl=$\teta_7$]}
  }
  \lablebox{
    \bar C^{1,7}&=1       \quad& \bar C^{4,7}&= 3       \quad& \bar C^{7,7}&=\tfrac{3}{2} \\ 
    \bar C^{2,7}&=\tfrac{3}{2} & \bar C^{5,7}&=\tfrac{5}{2}                               \\ 
    \bar C^{3,7}&=2            & \bar C^{6,7}&=2  
  }
\end{tikzpicture}

  \label{fig:aij}
\end{multicols}
\end{figure}

The Cartan matrices are the transposes of the transformation matrices describing the change of bases from the simple roots to the respective fundamental weights:
\begin{align*}
  \sigma &= \textstyle\sum_\nu C_{\sigma\nu}\omega_\nu    \quad(\text{for } \sigma,\nu \in\simpleR)\\
  \teta &= \textstyle \sum_\nu \bar C_{\teta\nu}\bar\omega_\nu \quad(\text{for } \teta,\nu \in\thetaR)
\end{align*}
We can therefore express the \((\thetaR^\vee)^\perp_\Q\)-component of a fundamental weight \(\omega_\teta\)
with \(\teta\in\thetaR\) in terms of these matrices.

\newcommand*{\neighbour}[1]{{\teta_{#1}}}
\begin{lem}\label{dyn:lem:omega}
  Let \(\omega_\teta\) with \(\teta\in\thetaR\) be a fundamental weight of \(G\),
  and let \(\thetaR_\teta\) denote the  connected component of \(\thetaR\) containing \(\teta\).\footnote{%
    This is intentional abuse of language.
    To be more precise, we would have to write:
    ``Let \(\thetaR_\teta\) denote the subset of \(\thetaR\)
    that corresponds to the connected component of the Dynkin diagram of \(\RSystem[\thetaR]\)
    that contains the node corresponding to \(\teta\).''
  }
  Write the projection of \(\omega_\teta\) onto \((\thetaR^\vee)^\perp\) in the basis \(\left\{\omega_\beta\mid\beta\in\nthetaR\right\}\):
  \[
  \prii(\omega_\teta) = \textstyle\sum_{\beta\in\nthetaR} c_{\teta\beta}\omega_\beta
  \]
  Then the coefficients \(c_{\teta\beta}\) are given by
  \[
  c_{\teta\beta} = \begin{cases}
    -\bar C^{\teta\teta_\beta} C_{\neighbour{\beta}\beta} & \text{if \(\beta\) has a neighbour \(\teta_\beta\) in \(\thetaR_\teta\)}\\
    0 & \text{otherwise}
  \end{cases}
  \]
  Here and in the following, we say that \(\teta_\beta\) is a neighbour of \(\beta\) if these roots are connected by an edge of the ambient Dynkin diagram.
\end{lem}

\begin{proof}
  We can write \(\omega_\teta\) as
  \(\prii(\omega_\teta) + \bar\omega_\teta
  = \textstyle\sum_{\beta\in\nthetaR} c_{\teta\beta}\omega_\beta + \textstyle\sum_{\teta'\in\theta}\bar C^{\teta\teta'}\teta'
  \).
  The result is obtained by applying \(\pairing{-}{\beta^\vee}\) to this equation and noting that \(C_{\teta'\beta}\) is non-zero only when
  \(\teta'\) is a neighbour of \(\beta\).
\end{proof}

\begin{cor}\label{dyn:cor:omega}\label{dyn:cor:omega}
  Let \(\thetaR_\teta\) denote the connected component of \(\thetaR\) containing \(\teta\). If we write the twist of \(\omega_\teta\) as
  \[
  \twist{\omega_\teta}=\sum_{\beta\in\nthetaR} m_\beta^\teta \omega_\beta,
  \]
  then the coefficients \(m_\beta^\teta\) are given by
  \[
  m_\beta^\teta = \begin{cases}
    (\bar C^{\teta\neighbour{\beta}}+\bar C^{\teta^\ddual\neighbour{\beta}})C_{\neighbour{\beta}\beta} & \text{ if \(\beta\) has a neighbour \(\neighbour{\beta}\) in \(\thetaR_\teta\) } \\
    0 & \text{ otherwise}
  \end{cases}
  \]
  The values of \(\bar C^{\teta\neighbour{\beta}}+\bar C^{\teta^\ddual\neighbour{\beta}}\) are displayed in \Cref{fig:twists}.
  The values \(C_{\neighbour{\beta}\beta}\) may be determined using \Cref{fig:Cartan}. (Of course, only the first five cases are relevant.)
  \qed
\end{cor}

\ifdraft{}{
\begin{figure}
  \begin{center}
\begin{tikzpicture}[dynPicture,auto]
  \dynTypeL{$A_l$}
  {[dynDiagram]
    \dynR{[right=of type][lbl=\dynweaktext{$\beta$}][weak]}
    {[dynTheta]
      \dynR[by weak line]{[lbl=$1$][marked] (1) }
      \dynR{[lbl=$1$] (a) }
      \dynR{[lbl=$1$]}
      \dynR{[lbl=$1$]}
      \dynR[by fadeout line]{[empty]} 
      \dynR[by fadein line]{[lbl=$1$] }
      \dynR{[lbl=$1$]}
      \dynR{[lbl=$1$] (z) }
      \dynR{[lbl=$1$]}
    }}
  \draw[dynduality] (a) to [bend left=160] (z);
  
  {[dynDiagram]
    {[dynTheta]
      \dynR{[below=\Platz of 1][lbl=$1$] }
      \dynR{[lbl=$2$][marked] (aa) }
      \dynR[with aa by weak line]{[above=of aa][weak]}
      \dynR[with aa by line]{[right=of aa][lbl=$2$]}
      \dynR{[lbl=$2$]}
      \dynR[by fadeout line]{[empty]} 
      \dynR[by fadein line]{[lbl=$2$] }
      \dynR{[lbl=$2$] }
      \dynR{[lbl=$2$] (zz) }
      \dynR{[lbl=$1$] (a2end) }
    }
    \node[right=of a2end]{${}_{(l\geq 3)}$};
  }
  \draw[dynduality] (aa) to [bend left=160] (zz);

  {[dynDiagram]
    {[dynTheta]
      \dynR{[below=2\Platz of 1][lbl=$1$] }
      \dynR{[lbl=$2$] (aaa) }
      \dynR{[lbl=$3$][marked] (3) }
      \dynR[with 3 by weak line]{[above=of 3][weak]}
      \dynR[with 3 by line]{[right=of 3][lbl=$3$]}
      \dynR[by fadeout line]{[empty]} 
      \dynR[by fadein line]{[lbl=$3$] }
      \dynR{[lbl=$3$] }
      \dynR{[lbl=$2$] (zzz) }
      \dynR{[lbl=$1$] (a3end) }
    }
    \node[right=of a3end]{${}_{(l\geq 5)}$};
  }
  \draw[dynduality] (aaa) to [bend left=160] (zzz);
\end{tikzpicture}
\end{center}

\begin{multicols}{2}
\begin{tikzpicture}[dynPicture]
  \dynTypeL{$B_l$}
  {[dynDiagram]
    \dynR{[right=of type][lbl=\dynweaktext{$\beta$}][weak]}
    {[dynTheta]
      \dynR[by weak line]{[lbl=$2$,marked]}
      \dynR{[lbl=$2$]}
      \dynR{[lbl=$2$]}
      \dynR[by fadeout line]{[empty] (x)} 
      \dynR[by fadein line]{[lbl=$2$] }
      \dynR[by right arrow]{[lbl=$1$] } 
    }}
  {[dynDiagram]
    {[dynTheta]
      \dynR{[lbl=$2$][below=of x] }
      \dynR{[lbl=$4$] (y) }
      \dynR[by right arrow]{[lbl=$3$][marked]}
    }
    \dynR[by weak line]{[weak]}
  }
  {[dynDiagram]
    {[dynTheta]
      \dynR{[lbl=$2$][below=of y]}
      \dynR[by right arrow]{[lbl=$2$][marked]}
    }
    \dynR[by weak line]{[weak]}
  }
\end{tikzpicture} 

\begin{tikzpicture}[dynPicture]
  \dynTypeL{$C_l$}
  {[dynDiagram]
    \dynR{[right=of type][lbl=\dynweaktext{$\beta$}][weak]}
    {[dynTheta]
      \dynR[by weak line]{[lbl=$2$,marked]}
      \dynR{[lbl=$2$]}
      \dynR{[lbl=$2$]}
      \dynR[by fadeout line]{[empty] (x) } 
      \dynR[by fadein line]{[lbl=$2$]}
      \dynR[by left arrow]{[lbl=$2$]} 
    }}
  {[dynDiagram]
    \dynR{[below=of x][weak] }
    {[dynTheta]
      \dynR[by weak line]{[lbl=$3$,marked]}
      \dynR[by right arrow]{[lbl=$2$]}
      \dynR{[lbl=$1$]}
    }}
\end{tikzpicture}

\begin{tikzpicture}[dynPicture]
  \dynTypeL{$E_6$}
  {[dynDiagram]
    {[dynTheta]
      \dynR{[right=of type][empty]}
      \dynR[by no line]{[lbl=$2$]}
      \dynR{[lbl=$3$] (a) }
      \dynR{[lbl=$4$](e6)}
      \dynR[with e6 by line]{[label={right:$2$}][above=of e6]}
      \dynR[with e6 by line]{[lbl=$3$][right=of e6] (z) }
      \dynR{[lbl=$2$,marked]}
    }
    \dynR[by weak line]{[label={right:\dynweaktext{$\beta$}}][weak]}
    \draw[dynduality] (a) to [bend left=120] (z);
  }
\end{tikzpicture}

\begin{tikzpicture}[dynPicture]
  \dynTypeL{$E_7$}
  {[dynDiagram]
    \dynR{[right=of type][empty]}
    {[dynTheta]
      \dynR[by no line]{[lbl=$2$]}
      \dynR{[lbl=$4$]}
      \dynR{[lbl=$6$](e6)}
      \dynR[with e6 by line]{[label={right:$3$}][above=of e6]}
      \dynR[with e6 by line]{[lbl=$5$][right=of e6]}
      \dynR{[lbl=$4$]}
      \dynR{[lbl=$3$,marked]}
    }
    \dynR[by weak line]{[label={right:\dynweaktext{$\beta$}}][weak]}
  }
\end{tikzpicture}

\begin{tikzpicture}[dynPicture]
  \dynTypeL{$D_l$}
  {[dynDiagram]
    \dynR{[weak][right=of type][lbl=\dynweaktext{$\beta$}] (0) }
    {[dynTheta]
      \dynR[by weak line]{[lbl=$2$,marked] (a) }
      \dynR{[lbl=$2$] }
      \dynR{[lbl=$2$] }
      \dynR[by fadeout line]{[empty]} 
      \dynR[by fadein line]{[lbl=$2\;\;$](v)}
      \dynR{[label={right:$1$}][above right=of v] (vv)}
      \dynR[with v]{[label={right:$1$}][below right=of v] (w)}
    }}
  \draw[dynduality,dashed] (vv) to [bend left=30]  node[auto] {${\;}_{(\text{if } l \text{ odd})}$} (w);

  {[dynDiagram]
    \dynR{[below=\Platz of 0][weak]}
    {[dynTheta]
      \dynR[by weak line]{[lbl=$3$,marked] (a)}
      \dynR{[lbl=$5$] (e6) }
      \dynR[with e6 by line]{[label={right:$3$}][above=of e6] (b) }
      \dynR[with e6 by line]{[lbl=$4$][right=of e6]}
      \dynR{[lbl=$3$]}
      \dynR{[lbl=$2$]}
      \dynR{[lbl=$1$]}
    }}
  \draw[dynduality] (a) to [bend left=30] (b);

  {[dynDiagram]
    \dynR{[below=2\Platz of 0][weak]}
    {[dynTheta]
      \dynR[by weak line]{[lbl=$3$,marked]}
      \dynR{[lbl=$4$] (e6) }
      \dynR[with e6 by line]{[label={right:$2$}][above=of e6]}
      \dynR[with e6 by line]{[lbl=$3$][right=of e6]}
      \dynR{[lbl=$2$]}
      \dynR{[lbl=$1$]}
    }}

  {[dynDiagram]
    \dynR{[below=3\Platz of 0][weak]}
    {[dynTheta]
      \dynR[by weak line]{[lbl=$2$,marked] (a) }
      \dynR{[lbl=$3$] (e6) }
      \dynR[with e6 by line]{[label={right:$2$}][above=of e6] (b)}
      \dynR[with e6 by line]{[lbl=$2$][right=of e6]}
      \dynR{[lbl=$1$]}
    }}
  \draw[dynduality] (a) to [bend left=30] (b);

  \node[below=3.8\Platz of 0]{};
\end{tikzpicture}
\end{multicols}
  \caption{The number underneath a root \ensuremath{\teta} indicates the value \ensuremath{\bar C^{\teta\neighbour{\beta}}+\bar C^{\teta^\ddual\neighbour{\beta}}}, where \(\neighbour{\beta}\) is the circled root in each diagram.
These values are easily computed from \Cref{fig:aij}.
The arrows indicate the involution \(\teta\mapsto\teta^\ddual\) on \(\thetaR\).
(Diagrams of types \(E_8\), \(F_4\) and \(G_2\) are not displayed as they cannot occur as proper subdiagrams.)
}
  \label{fig:twists}
\end{figure}
}

\begin{figure}
  \begin{multicols}{2}
    {
\renewcommand{\bar}{\overline}
\begin{tikzpicture}[dynPicture,auto]
  \dynTypeA{$A_{l,\; l \text{ odd}}$}
  {[dynDiagram]
    \dynR{[below=of type][weak][lbl=\dynweaktext{$\beta$}]}
    {[dynTheta]
      \dynR[by weak line]{[marked][lbl=\dynweaktext{$\teta_\beta$}]}
      \dynR{}
      \dynR[by fadeout line]{[empty]} 
      \dynR[by fadein line]{[lbl=$\bar 1$]}
      \dynR[by fadeout line]{[empty]} 
      \dynR[by fadein line]{}
      \dynR{}
    }}
\end{tikzpicture}
\begin{tikzpicture}[dynPicture]
  \dynTypeA{}
  {[dynDiagram]
     \dynR{[below=of type][empty]}
    {[dynTheta]
      \dynR[by no line]{ }
      \dynR{}
      \dynR{[marked] (3a) }
      \dynR[with 3a by weak line]{[above=of 3a][weak]}
      \dynR[with 3a by line]{[right=of 3a][lbl=$\bar 1$]}
      \dynR{} 
      \dynR{}
      \dynR{}
    }}
\end{tikzpicture}
\begin{tikzpicture}[dynPicture]
  \dynTypeA{}  
  {[dynDiagram]
    \dynR{[below=of type][empty]}
    {[dynTheta]
      \dynR[by no line]{ }
      \dynR{}
      \dynR{[marked][lbl=$\bar 1$] (3b) }
      \dynR[with 3b by weak line]{[above=of 3b][weak]}
      \dynR[with 3b by line]{[right=of 3b]}
      \dynR{}
    }}
\end{tikzpicture}

\begin{tikzpicture}[dynPicture]
  \dynTypeL{$B_l$}
  {[dynDiagram]
    \dynR{[right=of type][weak][lbl=\dynweaktext{$\beta$}]}
    {[dynTheta]
      \dynR[by weak line]{[marked][lbl=\dynweaktext{$\teta_\beta$}]}
      \dynR{(x)}
      \dynR{}
      \dynR[by fadeout line]{[empty]} 
      \dynR[by fadein line]{ }
      \dynR[by right arrow]{[lbl=$\bar 1$] } 
    }}
  {[dynDiagram]
    {[dynTheta]
      \dynR{[below=of x] }
      \dynR{ (y) }
      \dynR[by right arrow]{[lbl=$\bar 1$][marked]}
    }
    \dynR[by weak line]{[weak]}
  }
\end{tikzpicture}

\begin{tikzpicture}[dynPicture]
  \dynTypeL{$C_3$}
    {[dynDiagram]
    \dynR{[right=of type][weak][lbl=\dynweaktext{$\beta$}] }
    {[dynTheta]
      \dynR[by weak line]{[lbl=$\bar 1$][marked]}
      \dynR[by right arrow]{}
      \dynR{[lbl=$\bar 1$]}
    }}
\end{tikzpicture}

\begin{tikzpicture}[dynPicture]
  \dynTypeL{$E_7$}
  {[dynDiagram]
    {[dynTheta]
      \dynR{[right=of type]}
      \dynR{}
      \dynR{(e6)}
      \dynR[with e6 by line]{[label={right:$\bar 1$}][above=of e6]}
      \dynR[with e6 by line]{[lbl=$\bar 1$][right=of e6]}
      \dynR{ }
      \dynR{[lbl=$\bar 1$,marked]}
    }
    \dynR[by weak line]{[label={right:\dynweaktext{$\beta$}}][weak]}
  }
\end{tikzpicture}
\bigskip

\begin{tikzpicture}[dynPicture]
  \dynTypeA{$D_{l,\; l \text{ even}}$}
  {[dynDiagram]
    \dynR{[weak][below=of type][lbl=\dynweaktext{$\beta$}] (0) }
    {[dynTheta]
      \dynR[by weak line]{[marked][lbl=\dynweaktext{$\teta_\beta$}] (a) }
      \dynR{ }
      \dynR{ }
      \dynR[by fadeout line]{[empty]} 
      \dynR[by fadein line]{(v)}
      \dynR{[label={right:$\bar 1$}][above right=of v] (vv)}
      \dynR[with v]{[label={right:$\bar 1$}][below right=of v] (w)}
    }}
\end{tikzpicture}

\begin{tikzpicture}[dynPicture]
  \dynTypeA{$D_7$}
  {[dynDiagram]
    \dynR{[below=of type][weak]}
    {[dynTheta]
      \dynR[by weak line]{[marked] (a)}
      \dynR{[lbl=$\bar 1$] (e6) }
      \dynR[with e6 by line]{[above=of e6] (b) }
      \dynR[with e6 by line]{[right=of e6]}
      \dynR{[lbl=$\bar 1$]}
      \dynR{}
      \dynR{[lbl=$\bar 1$]}
    }}
\end{tikzpicture}\\

\begin{tikzpicture}[dynPicture]
  \dynTypeA{$D_6$}
  {[dynDiagram]
    \dynR{[below=of type][weak]}
    {[dynTheta]
      \dynR[by weak line]{[lbl=$\bar 1$,marked]}
      \dynR{ (e6) }
      \dynR[with e6 by line]{[above=of e6]}
      \dynR[with e6 by line]{[lbl=$\bar 1$][right=of e6]}
      \dynR{}
      \dynR{[lbl=$\bar 1$]}
    }}
\end{tikzpicture}\\

\begin{tikzpicture}[dynPicture]
  \dynTypeA{$D_5$}
  {[dynDiagram]
    \dynR{[below=of type][weak]}
    {[dynTheta]
      \dynR[by weak line]{[marked] (a) }
      \dynR{[lbl=$\bar 1$] (e6) }
      \dynR[with e6 by line]{[above=of e6] (b)}
      \dynR[with e6 by line]{[right=of e6]}
      \dynR{[lbl=$\bar 1$]}
    }}
\end{tikzpicture}
}

  \captionsetup{margin=0pt}
  \caption{A simplified version of \Cref{fig:twists}, indicating only for which self-dual roots \ensuremath{\teta} the value \mbox{\ensuremath{2\bar C^{\teta\neighbour{\beta}} = \bar C^{\teta\neighbour{\beta}}+\bar C^{\teta^\ddual\neighbour{\beta}}}} is odd.
  }
  \label{fig:twists-simplified}
  \end{multicols}
\end{figure}

\subsection{The marking scheme}
\Cref{dyn:cor:omega} gives an explicit description of the twist vectors  \mbox{\(\vec m^\teta = (m_\beta^\teta)_{\beta\in\nthetaR} \)} for each root \(\teta\in\thetaR\).  However, we see from \Cref{vanishing-for-tate} that only a small part of this information is relevant for determining whether the twisted Witt groups\slash Tate cohomology groups of \(G/P\) vanish. In particular:
\begin{enumerate}[(i)]
\item \label{e:r:self-dual}
  We only need to know those \(\vec m^\teta\) for which \(\teta\) is self-dual (\(\teta^\ddual=\teta\)).
\item \label{e:r:two}
  We only need their reductions \(\bar{\vec m}^\teta\) modulo two.
\item \label{e:r:track}
  We do not need to keep track which twist vector belongs to which root \(\teta\):  we only need to know which twist vectors \(\bar{\vec m}^\teta\) \emph{occur} for the collection of all self-dual roots of \(\thetaR\).
\end{enumerate}

\ifdraft{}{
\begin{wrapfigure}[13]{R}{3cm}
  \centering
  \begin{tikzpicture}[dynPicture,y=0.6cm,node distance=8mm]
{[dynDiagram,on grid]
\node[on chain] (0) {$\teta$};
\node[on chain] {$\beta$};
\node[on chain] {$C_{\teta\beta}$};
}
{[dynDiagram,on grid]
\dynR{at (0,-1) [theta]}
\dynR{}
\node[on chain]{-1};
}
{[dynDiagram,on grid]
\dynR{at (0,-2) [theta]}
\dynR[by left arrow]{}
\node[on chain]{-1};
}
{[dynDiagram,on grid]
\dynR{at (0,-3) [theta]}
\dynR[by triple left arrow]{}
\node[on chain]{-1};
}
{[dynDiagram,on grid]
\dynR{at (0,-4) [theta]}
\dynR[by right arrow]{}
\node[on chain]{-2};
}
{[dynDiagram,on grid]
\dynR{at (0,-5) [theta]}
\dynR[by triple right arrow]{}
\node[on chain]{-3};
}
{[dynDiagram,on grid]
\dynR{at (0,-6) [theta]}
\dynR[by no line]{}
\node[on chain]{~0};
}
{[dynDiagram,on grid]
\dynR{at (0,-7) [theta]}
\node[on chain]{${}_{(\teta=\beta)}$};
\node[on chain]{~2};
}
\end{tikzpicture}

  \captionsetup{justification=centering}
  \caption{The values \ensuremath{C_{\teta\beta}}}\label{fig:Cartan}
\end{wrapfigure}
}

Given (i) and (ii), we may simplify \Cref{fig:twists} to \Cref{fig:twists-simplified}.
By (iii),  we can summarize the relevant information by decorating the Dynkin diagram of \(G/P\): for each self-dual root \(\teta\) of \(\thetaR\), we add a mark connecting all those neighbours of \(\thetaR_\teta\) for which \(\bar m^\teta_\beta\) is non-zero. That is, if \(\bar{\vec m}^\teta\) has exactly one non-zero entry, we simply mark the corresponding neighbour in \(\simpleR\); if \(\bar{\vec m}^\teta\) has several non-zero entries, we add a mark that connects all the corresponding neighbours.

\Cref{dyn:cor:omega} implies that the marks of our Dynkin diagrams only depend on the types of the connected components of \(\thetaR\) and on the edges by which these components are connected to their neighbours in \(\simpleR\).
Thus, as claimed in \Cref{sec:marking-scheme}, it suffices to consider connected components of \(\thetaR\) in our marking scheme.
The precise rules are now easily derived, as we illustrate with a few examples.

First, consider the case when \(\thetaR\subset\simpleR\) is of the form \(A_l\subset A_n\), and suppose \(A_l\) has neighbours \(\beta_1\) and \(\beta_2\) on either side:
\[\ifdraft{}{
    \begin{tikzpicture}[dynPicture,auto]
      {[dynDiagram]
        \dynR{[empty]}
        \dynR[by fadein line]{}
        \dynR{[label={above:$\beta_1$}]}
        {[dynTheta]
          \dynR{[label={above:$\neighbour{\beta_1}$}] }
          \dynR{}
          \dynR[by fadeout line]{[empty]}
          \dynR[by fadein line]{[label={above:$\teta$}] }
          \dynR[by fadeout line]{[empty]}
          \dynR[by fadein line]{}
          \dynR{[label={above:$\neighbour{\beta_2}$}] }
        }
        \dynR{[label={above:$\beta_2$}]}
        \dynR{}
        \dynR[by fadeout line]{[empty]}
      }
    \end{tikzpicture}
  }
\]
If \(l\) is even, there are no self-dual roots in \(A_l\), so no roots get marked.
If \(l\) is odd, we have exactly one self-dual root \(\teta := \teta_{\nicefrac{(l+1)}{2}}\) in \(A_l\).
We see from \Cref{fig:twists-simplified} that
\(2\bar C^{\teta\neighbour{\beta_1}} = \bar C^{\teta\neighbour{\beta_1}}+\bar C^{\teta^\ddual\neighbour{\beta_1}}\) and
\(2\bar C^{\teta\neighbour{\beta_2}}\) are odd.
Moreover, since \(\beta_1\) and \(\beta_2\) are connected to \(A_l\) by a single line,
\(C_{\neighbour\beta_1\beta_1} = C_{\neighbour\beta_2\beta_2} = -1\).
Thus, both \( m^\teta_{\beta_1}\) and \( m^\teta_{\beta_2}\) are odd,
and we need to add a mark to our diagram connecting \(\beta_1\) and \(\beta_2\):
\[\ifdraft{}{
  \begin{tikzpicture}[dynPicture,auto]
    {[dynDiagram]
      \dynR{[empty]}
      \dynR[by fadein line]{}
      \dynR{(b1)}
      {[dynTheta]
        \dynR{}
        \dynR{}
        \dynR[by fadeout line]{[empty]}
        \dynR[by fadein line]{}
        \dynR[by fadeout line]{[empty]}
        \dynR[by fadein line]{}
        \dynR{}
      }
      \dynR{(b2)}
      \dynR{}
      \dynR[by fadeout line]{[empty]}
    }
    \dyndoubletwist{b1}{b2}
  \end{tikzpicture}
}\]

The same reasoning works for \(A_l\subset B_n\),
except that in this case it can happen that the neighbour \(\beta_2\) is connected to \(A_l\)
via a double line with an arrow pointing away from \(A_l\).
In that case, we find that \(C_{\neighbour\beta_2\beta_2} = -2\),
so \(m^\teta_{\beta_2}\) is even and only \(\beta_1\) gets marked.
Thus, the example from further above should be decorated with the following marks:
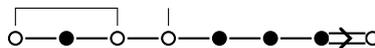
\begin{center}
  \ifdraft{}{
    \begin{tikzpicture}[dynPicture]
      {[dynDiagram]
        \dynR{ (b1)}
        {[dynTheta]
          \dynR{}
        }
        \dynR{ (b2)}
        \dynR{ (b) }
        {[dynTheta]
          \dynR{}
          \dynR{}
          \dynR{}
        }
        \dynR[by right arrow]{}
      }
      \dyndoubletwist{b1}{b2}
      \dyntwist{b}
    \end{tikzpicture}
  }
\end{center}

As another example, consider the case of \(B_l\subset B_n\). Let \(\beta\) be the unique neighbour of \(B_l\):
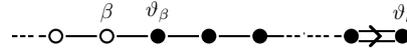
\[\ifdraft{}{
  \begin{tikzpicture}[dynPicture]
    {[dynDiagram]
      \dynR{[empty]}
      \dynR[by fadein line]{}
      \dynR{[label={above:$\beta$}]}
      {[dynTheta]
        \dynR{[label={above:$\neighbour\beta$}]}
        \dynR{}
        \dynR{}
        \dynR[by fadeout line]{[empty]}
        \dynR[by fadein line]{}
        \dynR[by right arrow]{[label={above:$\teta_l$}]}
      }}
  \end{tikzpicture}
}\]
All roots of \(B_l\) are self-dual. However, we see from \Cref{fig:twists-simplified} that \(2\bar C^{\teta\neighbour\beta}\) is odd only for the shortest root $\teta_l$ of \(B_l\). The value of \(C_{\neighbour\beta\beta}\) is once again \(-1\), so we find that \(m^{\teta_l}_\beta=2\bar C^{\teta_l\neighbour\beta}C_{\neighbour\beta\beta}\) is odd and \(\beta\) gets marked:
\[\ifdraft{}{
  \begin{tikzpicture}[dynPicture]
    {[dynDiagram]
      \dynR{[empty]}
      \dynR[by fadein line]{}
      \dynR{(b)}
      {[dynTheta]
        \dynR{}
        \dynR{}
        \dynR{}
        \dynR[by fadeout line]{[empty]}
        \dynR[by fadein line]{}
        \dynR[by right arrow]{}
      }}
    \dyntwist{b}
  \end{tikzpicture}
}\]

\end{document}